%% file: Lsc_Cheeger_2.tex
\documentclass[12pt]{article}
\usepackage{a4wide,color,eucal,enumerate,mathrsfs}
\usepackage[normalem]{ulem}
\usepackage{amsmath,amssymb,epsfig,graphicx}
\usepackage{esint}
\numberwithin{equation}{section}

\usepackage[latin1]{inputenc}

\input macro2010
\input macro_diffusion

\newcommand{\relgradq}[2]{|\nabla #1|_{*,#2}} 
\newcommand{\weakgradq}[2]{|\nabla #1|_{w,#2}} 

\newenvironment{proof}{\removelastskip\par\medskip   
\noindent{\em Proof.}
\rm}{\penalty-20\null\hfill$\square$\par\medbreak}

\newtheorem{theorem}{Theorem}[section]

\newtheorem{corollary}[theorem]{Corollary}
\newtheorem{lemma}[theorem]{Lemma}
\newtheorem{proposition}[theorem]{Proposition}
\newtheorem{definition}[theorem]{Definition}

\newtheorem{remark}[theorem]{Remark}

\newcommand{\prob}[1]{\mathscr P(#1)}

\newcommand{\lims}{\varlimsup}

\newcommand{\e}{{\rm{e}}}                           
\newcommand{\fr}{\hfill$\blacksquare$}                      

\renewcommand{\F}{\mathcal{F}_q}
\newcommand{\Fh}[1]{\mathcal{F}_{#1}}
\newcommand{\Ftwo}{\mathcal{F}_2}

\newcommand{\Ch}{{\rm Ch}}
\newcommand{\Dh}[1]{\mathcal{D}_{#1}}
\newcommand{\Dloc}{\mathcal{D}_{p, {\rm loc}}}

\newcommand{\Ph}[1]{\mathcal{P}_{#1}}
\newcommand{\Lipa}{\Lip_a}
\newcommand{\PCh}[1]{\mathcal{PC}_{#1}}
\newcommand{\Nh}{\mathcal{N}_h}
\newcommand{\Nn}{\mathcal{N}}
\newcommand{\Nph}[1]{\mathcal{N}_{q,#1}}
\newcommand{\ep}{\varepsilon}

\renewcommand{\mm}{\mathfrak m}

\title{Sobolev spaces in metric measure spaces: \\ reflexivity and lower semicontinuity of slope}
\author{Luigi Ambrosio, Maria Colombo, Simone Di Marino}
\begin{document}

\maketitle
\begin{abstract}
In this paper we make a survey of some recent developments of the theory of Sobolev spaces $W^{1,q}(X,\sfd,\mm)$,
$1<q<\infty$, in metric measure spaces $(X,\sfd,\mm)$. In the final part of the paper we provide a new proof of the reflexivity of the 
Sobolev space based on $\Gamma$-convergence; this result extends Cheeger's work because no Poincar\'e inequality is needed
and the measure-theoretic doubling property is weakened to the metric doubling property of the support of $\mm$.
We also discuss the lower semicontinuity of the slope of Lipschitz functions and some open problems.
\end{abstract}

\tableofcontents
\section{Introduction}

This paper is devoted to the theory of Sobolev spaces $W^{1,q}(X,\sfd,\mm)$ on metric measure spaces $(X,\sfd,\mm)$.
It is on one hand a survey paper on the most recent developments of the theory occurred in \cite{Ambrosio-Gigli-Savare11},
\cite{Ambrosio-Gigli-Savare12} (see also \cite{AmbrosioDiMarino12} for analogous results in the space $BV$ of functions of bounded variation), but 
it contains also new results on the reflexivity of $W^{1,q}$, $1<q<\infty$, improving those of \cite{Cheeger00}.
The occasion for writing this paper has been the course given by the first author in Sapporo (July-August 2012).

In a seminal paper \cite{Cheeger00}, Cheeger investigated the fine properties of Sobolev functions on metric measure
spaces, with the main aim of providing generalized versions of Rademacher's theorem and, along with it, a description
of the cotangent bundle. Assuming that the Polish metric measure structure $(X,\sfd,\mm)$ is doubling and satisfies
a Poincar\'e inequality (see Definition~\ref{dfn:Doubling} and Definition~\ref{dfn:PI} for precise formulations of these structural assumptions)
he proved that the Sobolev spaces are reflexive and that the $q$-power of the slope is $L^q(X,\mm)$-lower semicontinuous,
namely
\begin{equation}\label{lscChee}
f_h,\,f\in\Lip(X),\,\,\,\int_X|f_h-f|^q\,\\d\mm\to 0
\quad\Longrightarrow\quad\liminf_{h\to\infty}\int_X|\nabla f_h|^q\,\d\mm\geq\int_X|\nabla f|^q\,\d\mm.
\end{equation}
Here the slope $|\nabla f|$, also called local Lipschitz constant, is defined by
$$
|\nabla f|(x):=\limsup_{y\to x}\frac{|f(y)-f(x)|}{\sfd(y,x)}.
$$
These results come as a byproduct of a generalized Rademacher's theorem, which can be stated as follows: there exist
an integer $N$, depending on the doubling and Poincar\'e constants, a Borel partition $\{X_i\}_{i\in I}$ of $X$ and 
Lipschitz functions $f^i_j$, $1\leq j\leq N(i)\leq N$, with the property that for all $f\in\Lip(X)$ it is possible to
find Borel coefficients $c^i_j$, $1\leq j\leq N$, uniquely determined $\mm$-a.e. on $X_i$, satisfying
\begin{equation}\label{cheeasy}
|\nabla (f-\sum_{j=1}^{N(i)}c_j^i(x)f_j^i)|(x)=0\qquad\text{$\mm$-a.e. on $X_i$.}
\end{equation}
It turns out that the family of norms on $\R^{N(i)}$ 
$$
\|(c_1^i,\ldots,c_{N(i)}^i)\|_x:=|\nabla \sum_{j=1}^{N(i)}c_j^if_j^i|(x)
$$
indexed by $x\in X_i$ satisfies, thanks to \eqref{cheeasy},
$$
\|(c_1^i(x),\ldots,c_{N(i)}^i(x))\|_x=|\nabla f|(x) \qquad\text{$\mm$-a.e. on $X_i$.}
$$
Therefore, this family of norms provides the norm on the cotangent bundle on $X_i$. Since $N(i)\leq N$, using for instance John's lemma one can find Hilbertian
equivalent norms $|\cdot |_x$ with bi-Lipschitz constant depending only on $N$. This leads to an equivalent (but not canonical)
Hilbertian norm and then to reflexivity. In this paper we aim mostly at lower semicontinuity and reflexivity: we recover the latter 
(and separability as well) without assuming the validity of the Poincar\'e inequality and replacing the doubling assumption on $(X,\sfd,\mm)$ 
with a weaker assumption, namely the geometric doubling of $(\supp\mm,\sfd)$. 

Sobolev spaces, as well as a weak notion of norm of the gradient $|\nabla f|_{C,q}$, are built in \cite{Cheeger00} by considering the best
possible approximation of $f$ by functions $f_n$ having a $q$-integrable upper gradient $g_n$, namely pairs $(f_n,g_n)$
satisfying
\begin{equation}\label{eq:strongug}
|f_n(\gamma_1)-f_n(\gamma_0)|\leq\int_\gamma g_n
\qquad\text{for all absolutely continuous curves $\gamma:[0,1]\to X$.}
\end{equation}
Here, by best approximation
we mean that we minimize 
$$
\liminf_{n\to\infty}\int_X|g_n|^q\,\d\mm
$$
among all sequences $f_n$ that converge to $f$ in $L^q(X,\mm)$. It must be emphasized that even though 
the implication \eqref{lscChee} does not involve at all weak gradients, its proof requires a fine analysis
of the Sobolev spaces and, in particular, their reflexivity. At the same time, in \cite{Shanmugalingam00}
this approach was proved to be equivalent to the one based on the theory of $q$-upper gradients
introduced in \cite{Koskela-MacManus} and leading to a gradient that we shall denote $|\nabla f|_{S,q}$.
 In this theory one imposes the validity of \eqref{eq:strongug}
on ``almost all curves'' in the sense of \cite{Fuglede} and uses this property to define $|\nabla f|_{S,q}$. 
Both approaches are described more in detail
in Appendix A of this paper (see also \cite{Heinonen07} for a nice account of the theory).

More recently, the first author, N.Gigli and G.Savar\'e developed, motivated by a research program on
metric measure spaces with Ricci curvature bounds from below, a new approach to calculus in metric
measure spaces (see also \cite{Gigli12} for the most recent developments). In particular, in  \cite{Ambrosio-Gigli-Savare11} 
and \cite{Ambrosio-Gigli-Savare12} Sobolev spaces and weak gradients are built by a slightly different relaxation procedure, involving
Lipschitz functions $f_n$ with bounded support and their slopes $|\nabla f_n|$ instead of functions $f_n$ with $q$-integrable
upper gradient $g_n$: this leads to a weak gradient a priori larger than $|\nabla f|_{C,q}$. Still in 
 \cite{Ambrosio-Gigli-Savare11}  and \cite{Ambrosio-Gigli-Savare12},
connection with the upper gradient point of view, a different notion of negligible set of curves 
(sensitive to the parametrization of the curves) to quantify exceptions 
in \eqref{eq:strongug} was introduced, leading to a gradient a priori smaller than $|\nabla f|_{S,q}$. One of the main results
of these papers is that all the four notions of gradient a posteriori coincide, and this fact is independent of doubling
and Poincar\'e assumptions. 

The paper, that as we said must be conceived mostly as a survey paper until Section~7, is organized as follows. In Section~2 we recall some
preliminary tools of analysis in metric spaces, the theory of gradient flows (which plays, via energy dissipation estimates,
a key role), $\Gamma$-convergence, $p$-th Wasserstein distance $W_p$, with $p$ dual to the Sobolev exponent $q$, and optimal transport theory.
The latter plays a fundamental role in the construction of suitable measures in the space of absolutely continuous
curves via the so-called superposition principle, that allows to pass from an
``Eulerian'' formulation (i.e. in terms of a curve of measures or a
curve of probability densities) to a ``Lagrangian'' one. In Section~3 we study, following
very closely \cite{Ambrosio-Gigli-Savare12}, the pointwise properties of
the Hopf-Lax semigroup
$$
Q_tf(x):=\inf_{y\in X}f(y)+\frac{\sfd^p(x,y)}{pt^{p-1}},
$$
also emphasizing the role of the so-called asymptotic Lipschitz constant
$$
{\rm Lip}_a(f,x):=\inf_{r>0}{\rm Lip}\bigl(f, B(x,r)\bigr)=\lim_{r\downarrow 0}{\rm Lip}\bigl(f, B(x,r)\bigr),
$$
which is always larger than $|\nabla f|$ and coincides with the upper semicontinuous relaxation of
$|\nabla f|$ in length spaces.

Section~4 presents the two weak gradients $|\nabla f|_{*,q}$ and $|\nabla f|_{w,q}$,
the former obtained by a relaxation and the latter by a weak upper gradient property.
As suggested in the final section of \cite{Ambrosio-Gigli-Savare12}, we work with an
even stronger (a priori) gradient, where in the relaxation procedure we replace $|\nabla f_n|$
with $\Lip_a(f_n,\cdot)$. We present basic calculus rules and stability properties of these
weak gradients.

Section~5 contains the basic facts we shall need on the gradient flow in $L^2(X,\mm)$ of the lower
semicontinuous functional $f\mapsto {\bf C}_q(f):=\tfrac{1}{q}\int_X|\nabla f|_{*,q}^q\,\d\mm$, in particular the entropy
dissipation rate
$$
\frac{\d}{\d t}\int_X\Phi(f_t)\,\d\mm=-\int_X\Phi''(f_t)|\nabla
f_t|^q_{*,q}\,\d\mm
$$
along this gradient flow. Notice that, in order to apply the Hilbertian theory of gradient flows, we need
to work in $L^2(X,\mm)$. Even when $\mm$ is finite, this requires a suitable definition (obtained by
truncation) of $|\nabla f|_{*,q}$ when $q>2$ and $f\in L^2\setminus L^q(X,\mm)$.

In Section~6 we prove the equivalence of gradients. Starting from a
function $f$ with $|\nabla f|_{w,q}\in L^q(X,\mm)$ we approximate it
by the gradient flow of $f_t$ of ${\bf C}_q$ starting from $f$ and we use
the weak upper gradient property to get
$$
\limsup_{t\downarrow 0}\frac1t\int_0^t\int_X\frac{|\nabla
f_s|_{*,q}^q}{f_s^{p-1}}\,\d\mm\d s\leq\int_X\frac{|\nabla
f|_{w,q}^q}{f^{p-1}}\,\d\mm
$$
where $p=q/(q-1)$ is the dual exponent of $q$. Using the stability
properties of $|\nabla f|_{*,q}$ we eventually get $|\nabla
f|_{*,q}\leq|\nabla f|_{w,q}$ $\mm$-a.e. in $X$.

In Section~7 we prove that the Sobolev space $W^{1,q}(X,\sfd,\mm)$ is reflexive when $1<q<\infty$, $(\supp\mm,\sfd)$ is separable and
doubling, and $\mm$ is finite on bounded sets. Instead of
looking for an equivalent Hilbertian norm (whose existence is presently known only if the metric measure structure is doubling and the
Poincar\'e inequality holds), we rather look for a discrete scheme, involving functionals $\Fh{\delta}(f)$ of the form
$$
\Fh{\delta}(f)=\sum_i \frac 1{\delta^q} \sum_{A^\delta_j \sim A^\delta_i} | f_{\delta,i} - f_{\delta,j}|^q\mm(A_i^\delta) .
$$
Here $A_i^\delta$ is a well chosen decomposition of $\supp\mm$ on scale $\delta$, $f_{\delta,i}=\fint_{A_i^\delta}f$ 
and the sum involves cells $A_j^\delta$ close to $A_i^\delta$,
in a suitable sense. This strategy is very close to the construction of approximate $q$-energies on fractal sets and more general
spaces, see for instance \cite{Peirone}, \cite{Sturm}.

It is fairly easy to show that any $\Gamma$-limit point $\Fh{0}$ of $\Fh{\delta}$ as $\delta\to 0$ satisfies
\begin{equation}\label{eq:gamma11}
\Fh{0}(f)\leq c(c_D,q)\int_X\Lip_a^q(f,\cdot)\,\d\mm\qquad\text{for all Lipschitz $f$ with bounded support,}
\end{equation}
where $c_D$ is the doubling constant of $(X,\sfd)$ (our proof gives $c(c_D,q)\leq 6^qc_D^3$). 
More delicate is the proof of lower bounds of $\Fh{0}$, which uses
a suitable discrete version of the weak upper gradient property and leads to the inequality
\begin{equation}\label{eq:gamma22}
\frac{1}{4^q}\int_X|\nabla f|_{w,q}^q\,\d\mm\leq\Fh{0}(f)\qquad\forall f\in W^{1,q}(X,\sfd,\mm).
\end{equation}
Combining \eqref{eq:gamma11}, \eqref{eq:gamma22} and the equivalence of weak gradients gives
$$
\frac{1}{4^q}\int_X|\nabla f|_{w,q}^q\,\d\mm\leq\Fh{0}(f)\leq c(c_D,q)\int_X|\nabla f|_{w,q}^q\,\d\mm
\qquad\forall f\in W^{1,q}(X,\sfd,\mm).
$$
The discrete functionals $\Fh{\delta}(f)+\sum_i |f_{\delta,i}|^q\mm(A_i^\delta)$ describe $L^q$ norms in suitable discrete spaces, hence they
satisfy the Clarkson inequalities; these inequalities (which reduce to the parallelogram identity in the
case $q=2$) are retained by the $\Gamma$-limit point $\Fh{0}+\|\cdot\|_q^q$. This leads to an equivalent 
uniformly convex norm in $W^{1,q}(X,\sfd,\mm)$, and therefore to reflexivity. As a byproduct one obtains
density of bounded Lipschitz functions in $W^{1,q}(X,\sfd,\mm)$ and separability. In this connection, notice that the results of
\cite{Ambrosio-Gigli-Savare11}, \cite{Ambrosio-Gigli-Savare12} provide, even without a doubling assumption,
a weaker property (but still sufficient for some applications), the so-called density in energy; on the other hand, 
under the assumptions of \cite{Cheeger00} one has even more, namely density of Lipschitz functions in the Lusin sense.\\*
Notice however that $\Fh{0}$, like the auxiliary Hilbertian norms of \cite{Cheeger00}, is not canonical:
it might depend on the decomposition $A_i^\delta$ and we don't expect the whole family $\Fh{\delta}$
to $\Gamma$-converge as $\delta\to0^+$.

In Section~8 we prove \eqref{lscChee}, following in large part the scheme of \cite{Cheeger00} (although we get the result
in a more direct way, without an intermediate result in length spaces). In particular we need the Poincar\'e inequality to establish the bound
$$ | \nabla f | \leq C \, |\nabla f |_{w,q} \qquad\text{for any Lipschitz function $f$ with bounded support,}$$
which, among other things, prevents $|\nabla f|_{w,q}$ from being trivial.

Finally, in the appendices we describe more in detail the intermediate gradients $|\nabla f|_{C,q}$ and $|\nabla f|_{S,q}$, we
provide another approximation by discrete gradients also in non-doubling spaces (but our results here are not conclusive)
and we list a few open problems.

\smallskip
\noindent {\bf Acknowledgement.} The first author acknowledges the support
of the ERC ADG GeMeThNES. The authors thank N.Gigli for useful comments on a
preliminary version of the paper.

\section{Preliminary notions}\label{sec:preliminary}

In this section we introduce some notation and recall a few basic
facts on absolutely continuous functions, gradient flows of convex
functionals and optimal transportation, see also
\cite{Ambrosio-Gigli-Savare08}, \cite{Villani09} as general
references.

\subsection{Absolutely continuous curves and slopes}
Let $(X,\sfd)$ be a metric space, $J\subset\R$ a closed interval and
$J\ni t\mapsto x_t\in X$. We say that $(x_t)$ is \emph{absolutely
continuous} if
$$
\sfd(x_s,x_t)\leq\int_s^tg(r)\,\d r\qquad\forall s,\,t\in J,\,\,s<t
$$
for some $g\in L^1(J)$. It turns out that, if $(x_t)$ is absolutely
continuous, there is a minimal function $g$ with this property,
called \emph{metric speed}, denoted by $|\dot{x}_t|$ and given for
a.e. $t\in J$ by
$$
|\dot{x}_t|=\lim_{s\to t}\frac{\sfd(x_s,x_t)}{|s-t|}.
$$
See \cite[Theorem~1.1.2]{Ambrosio-Gigli-Savare08} for the simple
proof.

We will denote by $C([0,1],X)$ the space of continuous curves from
$[0,1]$ to $(X,\sfd)$ endowed with the $\sup$ norm. The set
$AC^p([0,1],X)\subset C([0,1],X)$ consists of all absolutely
continuous curves $\gamma$ such that $\int_0^1|\dot\gamma_t|^p\,\d
t<\infty$: it is the countable union of the sets $\{\gamma:\
\int_0^1|\dot\gamma_t|^p\,\d t\leq n\}$, which are easily seen to be
closed if $p>1$. Thus $AC^p([0,1],X)$ is a Borel subset of
$C([0,1],X)$. The \emph{evaluation maps} $\e_t:C([0,1],X)\to X$ are
defined by
\[
\e_t(\gamma):=\gamma_t,
\]
and are clearly continuous.

Given $f:X\to\R$ and $E\subset X$, we denote by $\Lip (u, E)$ the Lipschitz constant of the function $u$ on $E$, namely
$$
\Lip(u,E):=\sup_{x,\,y\in E,\,x\neq y}\frac{|f(x)-f(y)|}{\sfd(x,y)} .
$$ 
Given $f:X\to\R$, we define \emph{slope} (also called local
Lipschitz constant) by
$$
|\nabla f|(x):=\lims_{y\to x}\frac{|f(y)-f(x)|}{\sfd(y,x)}.
$$
For $f,\,g:X\to\R$ Lipschitz it clearly holds
\begin{subequations}
\begin{align}
\label{eq:subadd}
|\nabla(\alpha f+\beta g)|&\leq|\alpha||\nabla f|+|\beta||\nabla g|\qquad\forall \alpha,\beta\in\R,\\
\label{eq:leibn} |\nabla (fg)|&\leq |f||\nabla g|+|g||\nabla f|.
\end{align}
\end{subequations}

We shall also need the following calculus lemma.

\begin{lemma}\label{lem:Fibonacci}
{
Let $f:(0,1)\to\R$, $q\in [1,\infty]$, $g\in L^q(0,1)$ nonnegative be satisfying
$$
|f(s)-f(t)|\leq\bigl|\int_s^t g(r)\,\d r\bigr|\qquad\text{for $\Leb{2}$-a.e. $(s,t)\in (0,1)^2$.}
$$
Then $f\in W^{1,q}(0,1)$ and $|f'|\leq g$ a.e. in $(0,1)$.}
\end{lemma}

\begin{proof} 
Let $N\subset (0,1)^2$ be the $\Leb{2}$-negligible subset where the above 
inequality fails. Choosing $s\in (0,1)$, whose existence is ensured by Fubini's theorem, such that
$(s,t)\notin N$ for a.e. $t\in (0,1)$, we obtain that $f\in L^\infty(0,1)$. 
Since the set $\{(t,h)\in (0,1)^2:\ (t,t+h)\in N\cap (0,1)^2\}$
is $\Leb{2}$-negligible as well, we can apply Fubini's theorem to obtain that for a.e. $h$ it holds
$(t,t+h)\notin N$ for a.e. $t\in (0,1)$. Let $h_i\downarrow 0$ with this property and use
the identities
$$
\int_0^1f(t)\frac{\phi(t+h)-\phi(t)}{h}\,\d t=-\int_0^1\frac{f(t-h)-f(t)}{-h}\phi(t)\,\d t
$$
with $\phi\in C^1_c(0,1)$ and $h=h_i$ sufficiently small to get
$$
\biggl|\int_0^1f(t)\phi'(t)\,\d t\biggr|\leq\int_0^1g(t)|\phi(t)|\,\d t.
$$
It follows that the distributional derivative of $f$ is a signed
measure $\eta$ with finite total variation which satisfies
\begin{displaymath}
  -\int_0^1f\phi'\,\d t=\int_0^1 \phi\,\d\eta,\qquad
  \Bigl|\int_0^1 \phi\,\d\eta\Bigr|\le \int_0^1g|\phi|\,\d
t\quad\text{for every }\phi\in C^1_c(0,1);
\end{displaymath}
therefore $\eta$ is absolutely continuous with respect to the Lebesgue
measure with $|\eta|\le g\Leb 1$. 
This gives the $W^{1,1}(0,1)$ regularity and, at the same time, the inequality
$|f'|\leq g$ a.e. in $(0,1)$. The case $q>1$ immediately follows
by applying this inequality when $g\in L^q(0,1)$.
\end{proof}

Following \cite{Heinonen-Koskela98}, we say that a Borel function
$g:X\to [0,\infty]$ is an upper gradient of a Borel function $f:X\to\R$ if the
inequality
\begin{equation}\label{eq:uppergradient}
\biggl|\int_{\partial\gamma}f\biggr|\leq\int_\gamma g
\end{equation}
holds for all absolutely continuous curves $\gamma:[0,1]\to X$. Here
$\int_{\partial\gamma} f=f(\gamma_1)-f(\gamma_0)$, while
$\int_\gamma g=\int_0^1g(\gamma_s)|\dot\gamma_s|\,\d s$.

It is well-known and easy to check that the slope is an upper
gradient, for locally Lipschitz functions.

\subsection{Gradient flows of convex and lower semicontinuous functionals}

Let $H$ be an Hilbert space, $\Psi:H\to\R\cup\{+\infty\}$ convex and
lower semicontinuous and $D(\Psi)=\{\Psi<\infty\}$ its finiteness
domain. Recall that a gradient flow $x:(0,\infty)\to H$ of $\Psi$ is
a locally absolutely continuous map with values in $D(\Psi)$
satisfying
$$
-\frac{\d}{\dt}x_t\in\partial^-\Psi(x_t)\qquad\text{for a.e. $t\in
(0,\infty)$.}
$$
Here $\partial^-\Psi(x)$ is the subdifferential of $\Psi$, defined
at any $x\in D(\Psi)$ by
$$
\partial^-\Psi(x):=\left\{p\in H^*:\ \Psi(y)\geq\Psi(x)+\langle
p,y-x\rangle\,\,\forall y\in H\right\}.
$$

We shall use the fact that for all $x_0\in\overline{D(\Psi)}$ there
exists a unique gradient flow $x_t$ of $\Psi$ starting from $x_0$,
i.e. $x_t\to x_0$ as $t\downarrow 0$, and that $t\mapsto\Psi(x_t)$
is nonincreasing and locally absolutely continuous in $(0,\infty)$.
In addition, this unique solution exhibits a regularizing effect,
namely $-\tfrac{\d}{\d t}x_t$ is for a.e. $t\in (0,\infty)$ the
element of minimal norm in
$\partial^-\Psi(x_t)$. 

\subsection{The space $(\prob X,W_p)$ and the superposition principle}

Let $(X,\sfd)$ be a complete and separable metric space and $p\in
[1,\infty)$. We use the notation $\prob X$ for the set of all Borel
probability measures on $X$. Given $\mu,\,\nu\in\prob X$, we define
the Wasserstein (extended) distance $W_p(\mu,\nu)\in [0,\infty]$
between them as
\[
W_p^p(\mu,\nu):=\min\int \sfd^p(x,y)\,\d\ggamma(x,y).
\]
Here the minimization is made in the class $\Gamma(\mu,\nu)$ of all
probability measures $\ggamma$ on $X\times X$ such that
$\pi^1_ \sharp\ggamma=\mu$ and $\pi^2_ \sharp\ggamma=\nu$, where
$\pi^i:X\times X\to X$, $i=1,\,2$, are the coordinate projections
and $f_ \sharp:\prob{Y}\to\prob{Z}$ is the push-forward operator induced
by a Borel map $f:Y\to Z$.

An equivalent definition of $W_p$ comes from the dual formulation of
the transport problem:
\begin{equation}
\label{eq:dualitabase} \frac1pW_p^p(\mu,\nu)=\sup_{\psi\in{\rm
Lip}_b(X)}\int\psi\, \d\mu+\int \psi^c\,\d\nu.
\end{equation}
Here ${\rm Lip}_b(X)$ stands for the class of bounded Lipschitz functions and
the $c$-transform $\psi^c$ is defined by
\[
\psi^c(y):=\inf_{x\in X}\frac{\sfd^p(x,y)}p-\psi(x).
\]

We will need the following result, proved in \cite{Lisini07}: it
shows how to associate to an absolutely continuous curve $\mu_t$
w.r.t. $W_p$ a plan $\ppi\in\prob{C([0,1],X)}$ representing the
curve itself (see also \cite[Theorem~8.2.1]{Ambrosio-Gigli-Savare08}
for the Euclidean case).

\begin{proposition}[Superposition principle]\label{prop:lisini}
Let $(X,\sfd)$ be a complete and separable metric space with $\sfd$
bounded, $p\in (1,\infty)$ and let $\mu_t\in AC^p\bigl([0,T];(\prob X,W_p)\bigr)$.
Then there exists $\ppi\in\prob{C([0,1],X)}$,
concentrated on $AC^p([0,1],X)$, such that $(\e_t)_\sharp\ppi=\mu_t$
for any $t\in[0,T]$ and
\begin{equation}\label{eq:Lisini}
\int|\dot\gamma_t|^p\,\d\ppi(\gamma)=|\dot\mu_t|^p\qquad \text{for
a.e. $t\in [0,T]$.}
\end{equation}
\end{proposition}

\subsection{$\Gamma$-convergence}

\begin{definition} Let $(X,\sfd)$ be a metric space and let $F_h:X\to [-\infty,+\infty]$. 
We say that $F_h$ $\Gamma$-converge to $F:X\to [-\infty,+\infty]$ if:
\begin{enumerate}
\item[(a)] For every sequence $(u_h)\subset X$ convergent to $u\in X$ we have
$$F(u)\leq \liminf_{h\to \infty} F_h(u_h);$$
\item[(b)] For all $u\in X$ there exists a sequence $(u_n)\subset X$ such that 
$$F(u)\geq \limsup_{h\to \infty} F_h(u_h).$$
\end{enumerate}
\end{definition}

Sequences satisfying the second property are called ``recovery sequences''; whenever $\Gamma$-convergence occurs,
they obviously satisfy $\lim_h F_h(u_h)=F(u)$.

The following compactness property of $\Gamma$-convergence (see for instance \cite[Theorem 8.5]{DalMaso93}) 
is well-known.

\begin{proposition}\label{prop:Gammacompact}
If $(X,\sfd)$ is separable, any sequence of functionals $F_h:X\to [-\infty,+\infty]$ admits a $\Gamma$-convergent
subsequence.
\end{proposition}

We quickly sketch the proof, for the reader's convenience. If $\{U_i\}_{i\in\N}$ is a countable basis of open
sets of $(X,\sfd)$, we may extract a subsequence $h(k)$ such that $\alpha_i:=\lim_k\inf_{U_i}F_{h(k)}$ exists in $\overline{\R}$ for
all $i\in\N$. Then, it is easily seen that
$$
F(x):=\sup_{U_i\ni x}\alpha_i\qquad x\in X
$$
is the $\Gamma$-limit of $F_{h(k)}$.

We will also need an elementary stability property of uniformly convex (and quadratic as well) functionals under
$\Gamma$-convergence. Recall that a positively $1$-homogeneous function $\Nn$ on a vector space $V$ is uniformly convex with modulus $\omega$ 
if there exists a function $\omega:[0,\infty)\to [0,\infty)$ with $\omega>0$ on $(0,\infty)$ such that 
$$\Nn(u)=\Nn(v)=1\qquad\Longrightarrow\qquad\Nn\left(\frac{u+v}{2}\right) \leq 1-\omega(\Nn(u-v))$$
for all $u,\,v\in V$.

\begin{lemma}\label{lemma:unif-conv}
Let $V$ be a normed space with the induced metric structure and let 
$\omega:[0,\infty)\to [0,\infty)$ be continuous, nondecreasing, positive on $(0,\infty)$. 
Let $\Nh$ be uniformly convex positively $1$-homogeneous functions on $V$ with the same modulus $\omega$, $\Gamma$-convergent to
some function $\Nn$. Then $\Nn$ is positively $1$-homogeneous and uniformly convex with modulus $\omega$.
\end{lemma}
\begin{proof} The verification of $1$-homogeneity of $\Nn$ is trivial.
Let $u,\,v\in V$ which satisfy $\Nn(u)= \Nn(v)=1$. Let $(u_h)$ and $(v_h)$ be recovery sequences for $u$ and $v$ respectively,
so that both $\Nh(u_h)$ and $\Nh(v_h)$ converge to 1. Hence, $u_h'=u_h/\Nh(u_h)$ and $v_h'=v_h/\Nh(v_h)$ still converge to $u$ and $v$
respectively. By assumption
$$
\Nh\left(\frac{u_h'+v_h'}{2}\right)+\omega(\Nh(u_h'-v_h'))\leq 1. 
$$
Thanks to property (a) of $\Gamma$-convergence, the monotonicity and the continuity of $\omega$ and  the superadditivity
of $\liminf$ we get
\[
\begin{split}
 \Nn\left(\frac{u+v}{2}\right)+
\omega \left(\Nn(u-v) \right)
&\leq
\liminf_{h\to \infty}  \Nh\left(\frac{u_h'+v_h'}{2}\right)+
\omega\left(\liminf_{h\to \infty}\Nh(u_h'-v_h') \right) \\
&\leq \liminf_{h\to \infty}  \left(\Nh\left(\frac{u_h'+v_h'}{2}\right)+\omega(\Nh(u_h'-v_h')) \right)\leq 1.
\end{split}
\]
\end{proof}

 \subsection{Doubling metric measure spaces and maximal functions}

Recall that a metric space $(X,\sfd)$ is doubling if there exists a natural number $c_D$ such that every ball of radius $r$ can be covered by 
at most $c_D$ balls of halved radius $r/2$. While this condition will be sufficient to establish reflexivity of the Sobolev spaces, in the
proof of lower semicontinuity of slope we shall actually need a stronger condition, involving also the reference measure $\mm$:

\begin{definition}[Doubling m.m. spaces] \label{dfn:Doubling} The metric measure space $(X, \sfd, \mm)$ is doubling if there exists $\tilde{c}_D\geq 0$ such that
\begin{equation}\label{eqn:doubll}
\mm (B(x,2r)) \leq \tilde{c}_D \mm (B(x,r)) \qquad \forall x \in\supp\mm, \,\,r>0.
\end{equation}
\end{definition}
This condition is easily seen to be equivalent to the existence of two real positive numbers $ \alpha , \,\beta>0$ which depend only on 
$\tilde{c}_D$ such that
\begin{equation}\label{eqn:doubl}
\mm (B(x,r_1)) \leq \beta  \left( \frac {r_1}{r_2} \right)^\alpha \mm (B(y,r_2) )
\quad\text{whenever $B(y,r_2) \subset B(x,r_1)$, $r_2\leq r_1$, $y\in\supp\mm$.}
\end{equation}
Indeed, $B(x,r_1)\subset B(y,2r_1)$, hence $\mm(B(x,r_1))\leq c_D^k\mm(B(y,r_2))$, where $k$ is the smallest
integer such that $2r_1\leq 2^k r_2$.  Since $k\leq 2+\ln_2(r_1/r_2)$, we obtain \eqref{eqn:doubl} with
$\alpha=\ln_2c_D$ and $\beta=c_D^2$. 

Condition \eqref{eqn:doubl} is stronger than the metric doubling property, in the sense that $(\supp\mm,\sfd)$ is doubling whenever
$(X,\sfd,\mm)$ is. Indeed, given a ball $B(x,r)$, let us choose recursively points $x_i\in B(x,r)\cap\supp\mm$ with $\sfd(x_i,x_j)\geq r/2$,
and assume that this is possible for $i=1,\ldots,N$. Then, the balls $B(x_i,r/4)$ are disjoint and
$$
\mm\bigl(B\bigl(x_i,\frac r 4\bigr)\bigr)\geq \tilde{c}_D^{-3}\mm(B(x_i, 2r))\geq c_D^{-3}\mm(B(x,r)),
$$
so that $N\leq \tilde{c}_D^3$. It follows that $(X,\sfd)$ is doubling, with doubling constant $c_D\leq\tilde{c}_D^3$.
 Conversely (but we shall not need this fact) any compact doubling metric space supports a nontrivial doubling measure.

\begin{definition}[Local maximal function]\label{dfn:maxfun} Given $q\in [1,\infty)$, $\ep>0$ and 
a Borel function $f:X\to\R$ such that $|f|^q$ is $\mm$-integrable on bounded sets,
we define the $\ep$-maximal function 
$$M^{\ep}_q f (x):= \left( \sup_{ 0<r \leq \ep} \fint_{B(x,r)} |f|^q \,\d\mm \right)^{1/q}\qquad x\in\supp\mm.$$
\end{definition}

The function $M^{\ep}_q f (x)$ is nondecreasing w.r.t. $\ep$, moreover $M^{\ep}_q f (x) \searrow |f|(x)$ at any Lebesgue point
$x$ of $|f|^q$, namely a point $x\in\supp\mm$ satisfying
\begin{equation}\label{eqn:leb11} 
\lim_{r \downarrow 0} \frac 1{ \mm(B(x,r)) } \int_{ B(x,r)} |f(y)|^q\,\d\mm(y) = |f(x) |^q.
\end{equation}

We recall that, in doubling metric measure spaces (see for instance \cite{Heinonen07}), under the previous assumptions on $f$ we have that
$\mm$-a.e. point is a Lebesgue point of $|f|^q$ (the proof is based on the so-called Vitali covering lemma). By applying this
property to $|f-s|^q$ with $s\in\Q$ one even obtains
\begin{equation}\label{eqn:leb1} 
\lim_{r \downarrow 0} \frac 1{ \mm(B(x,r)) } \int_{ B(x,r)} |f(y)-f(x)|^q\,\d\mm(y) = 0
\end{equation}
for $\mm$-a.e. $x\in\supp\mm$. We shall need a further enforcement of the Lebesgue point property:

\begin{lemma}\label{lem:leb} Let $(X,\sfd ,\mm)$ be a doubling metric measure space and let $f:X\to\R$ be a Borel
function such that  $|f|^q$ is $\mm$-integrable on bounded sets. Then, at any point $x$ where \eqref{eqn:leb1} is
satisfied, it holds
\begin{equation}\label{eqn:leb2} 
\lim_{n \to \infty} \frac 1{\mm (E_n)} \int_{E_n} |f(y) - f(x) |^q \,\d\mm (y) = 0
\end{equation}
whenever $E_n\subset X$ are Borel sets satisfying $B(y_n,\tau r_n) \subset E_n \subseteq B(x,r_n)$
with $y_n\in\supp\mm$ and $r_n \to 0$, for some $\tau\in (0,1]$ independent of $n$.  In particular $\fint_{E_n} f \,\d\mm\to f(x)$.
\end{lemma}

\begin{proof}  Since $\mm$ is doubling we can use \eqref{eqn:doubl} to obtain
\begin{align*}
\frac 1{ \mm ( E_n) } \int_{E_n} |f(y) - f(x) |^q \,\d\mm (y) & \leq \frac 1{ \mm ( B(y_n,\tau r_n))} \int_{E_n} |f(y) - f(x) |^q \, \d\mm (y) \\
& \leq \frac 1{\mm(B(y_n,\tau r_n))} \int_{B(x,r_n)} |f(y) - f(x) |^p \, \sfd \mm (y) \\
& \leq \frac{\mm(B(x,r_n))}{\mm(B(y_n,\tau r_n))} \fint_{B_{r_n}(x)} |f(y) - f(x) |^p \, \sfd \mm (y) \\
& \leq \beta \tau^{-\alpha} \fint_{B(x,r_n)} |f(y) - f(x) |^q \, \sfd \mm (y). \\
\end{align*}
By \eqref{eqn:leb1} the last term goes to $0$ at $q$-Lebesgue points of $f$ and we proved \eqref{eqn:leb2}. Finally, by Jensen's inequality, 
$$ \left| \fint_{E_n} f \,\d\mm - f(x) \right|^q \leq\fint_{E_n} |f- f(x)|^q \, \d\mm\rightarrow 0. $$
\end{proof}

\section{Hopf-Lax formula and Hamilton-Jacobi equation}\label{sec:hopflax}

Aim of this section is to study the properties of the Hopf-Lax
formula in a metric space $(X,\sfd)$ and its relations with the
Hamilton-Jacobi equation. Notice that there is no reference
measure $\mm$ here and that not even completeness is needed for the
results of this section. We fix a power $p\in (1,\infty)$ and denote by
$q$ its dual exponent.

Let $f:X\to\R$ be a Lipschitz function. For $t>0$ define
\begin{equation}\label{eq:Nicola1}
F(t,x,y):=f(y)+\frac{\sfd^p(x,y)}{pt^{p-1}},
\end{equation}
and the function $Q_tf:X\to\R$ by
\begin{equation}\label{eq:Nicola2}
Q_tf(x):=\inf_{y\in X}F(t,x,y).
\end{equation}
Also, we introduce the functions $D^+,\,D^-:X\times(0,\infty)\to\R$
as
\begin{equation}\label{eq:defdpm}
\begin{split}
D^+(x,t)&:=\sup\, \limsup_{n\to\infty}\sfd(x,y_n),\\
D^-(x,t)&:=\inf\, \liminf_{n\to\infty}\sfd(x,y_n),\\
\end{split}
\end{equation}
where, in both cases, the sequences $(y_n)$ vary among all minimizing sequences for 
$F(t,x,\cdot)$. We also set $Q_0f=f$ and $D^\pm(x,0)=0$. Arguing as
in \cite[Lemma~3.1.2]{Ambrosio-Gigli-Savare08} it is easy to check
that the map $[0,\infty)\ni(t,x)\mapsto Q_tf(x)$ is continuous.
Furthermore, the fact that $f$ is Lipschitz easily yields
\begin{equation}
\label{eq:boundD} D^-(x,t)\leq D^+(x,t)\leq t(p\Lip(f))^{1/(p-1)}.
\end{equation}

\begin{proposition}[Monotonicity of $D^\pm$]\label{prop:dmon}
For all $x\in X$ it holds 
\begin{equation}\label{eq:basic_mono}
D^+(x,t)\leq D^-(x,s)\qquad 0\leq t< s.
\end{equation}
As a consequence, $D^+(x,\cdot)$ and $D^-(x,\cdot)$ are both
nondecreasing, and they coincide with at most countably many
exceptions in $[0,\infty)$.
\end{proposition}
\begin{proof} Fix $x\in X$. For $t=0$ there is nothing to prove. Now pick $0<t<s$
and for every $\ep>0$ choose $x_{t,\ep}$ and $x_{s,\ep}$ minimizers up to $\ep$ of $F(t,x,\cdot)$ and
$F(s,x,\cdot)$ respectively, namely such that  $F(t,x,x_{t,\ep}) - \ep \leq F(t,x, w) $ and $F(s,x,x_{s,\ep}) - \ep \leq F(s,x, w) $ for 
every $w\in X$. 
Let us assume that $\sfd(x,x_{t,\ep}) \geq D^+(x,t)-\ep$ and $
\sfd(x,x_{s,\ep})\leq D^-(x,s)+\ep$. The minimality up to $\ep$ of $x_{t,\ep},\,x_{s,\ep}$ gives
\[
\begin{split}
f(x_{t,\ep})+\frac{\sfd^p(x_{t,\ep},x)}{pt^{p-1}}&\leq f(x_{s,\ep})+\frac{\sfd^p(x_{s,\ep},x)}{pt^{p-1}} +\ep \\
f(x_{s,\ep})+\frac{\sfd^p(x_{s,\ep},x)}{ps^{p-1}}&\leq
f(x_{t,\ep})+\frac{\sfd^p(x_{t,\ep},x)}{ps^{p-1}}+\ep .
\end{split}
\]
Adding up and using the fact that $\tfrac1t\geq\tfrac 1s$ we deduce
\[
(D^+(x,t)-\ep)^p \leq \sfd^p(x_{t,\ep},x)\leq \sfd^p(x_{s,\ep},x)+\ep (t^{1-p}-s^{1-p})^{-1} \leq  (D^-(x,s)+\ep)^p + \ep (t^{1-p}-s^{1-p})^{-1}.
\]
Letting $\ep\to 0$ we obtain \eqref{eq:basic_mono}. Combining this with the inequality
$D^-\leq D^+$ we immediately obtain that both functions are
nonincreasing. At a point of right continuity of $D^-(x,\cdot)$ we
get
$$
D^+(x,t)\leq\inf_{s>t}D^-(x,s)=D^-(x,t).
$$
This implies that the two functions coincide out of a countable set.
\end{proof}

Next, we examine the semicontinuity properties of $D^\pm$. These
properties imply that points $(x,t)$ where the equality
$D^+(x,t)=D^-(x,t)$ occurs are continuity points for both $D^+$ and
$D^-$.

\begin{proposition}[Semicontinuity of $D^\pm$]
$D^+$ is upper semicontinuous and $D^-$ is lower semicontinuous in
$X\times [0,\infty)$.
\end{proposition}
\begin{proof}
We prove lower semicontinuity of $D^-$, the proof of upper
semicontinuity of $D^+$ being similar. Let $(x_i,t_i)$ be any
sequence converging to $(x,t)$ such that the limit of $D^-(x_i,t_i)$
exists and assume that $t>0$ (the case $t=0$ is trivial). For every
$i$, let $(y^n_i)$ be a minimizing sequence of $F(t_i,x_i,\cdot)$ for which
$\lim_n\sfd(y_i^n,x_i)=D^-(x_i,t_i)$, so that
\[
\lim_{n\to\infty} f(y^n_i)+\frac{\sfd^p(y^n_i,x_i)}{pt_i^{p-1}}=Q_{t_i}f(x_i).
\]
Using the continuity of $Q_t$ we get
\[
Q_t f(x)=
\lim_{i\to\infty} \lim_{n\to\infty} f(y_i^n)+\frac{\sfd^p(y^n_i,x_i)}{pt^{p-1}}
\geq\limsup_{i\to\infty} \limsup_{n\to\infty} f(y_i^n)+\frac{\sfd^p(y^n_i,x)}{pt^{p-1}}.
\]
Analogously
$$
\lim_{i\to\infty} D^-(x_i,t_i)=
\lim_{i\to\infty} \lim_{n\to\infty} \sfd (y_i^n,x_i)\geq
\limsup_{i\to\infty} \limsup_{n\to\infty} \sfd (y_i^n,x).
$$
Therefore by a diagonal argument we can find a minimizing sequence $(y_i^{n(i)})$ 
for $F(t,x,\cdot)$ with $\limsup_i\sfd(y_i^{n(i)},x)\leq\lim_iD^-(x_i,t_i)$, which gives the result.
\end{proof}

\begin{proposition}[Time derivative of
$Q_tf$]\label{prop:timederivative} The map $t\mapsto Q_tf$ is
Lipschitz from $[0,\infty)$ to $C(X)$ and, for all $x\in X$, it
satisfies: 
\begin{equation}\label{eq:Dini1}
\frac{\d}{\d
t}Q_tf(x)=-\frac{1}{q}\biggl[\frac{D^{\pm}(x,t)}{t}\biggr]^p,
\end{equation}
for any $t>0$, with at most countably many exceptions.
\end{proposition}
\begin{proof} Let $t<s$ and for every $\ep>0$ choose $x_{t,\ep}$ and $x_{s,\ep}$ minimizers up to $\ep$ of $F(t,x,\cdot)$ and
$F(s,x,\cdot)$ respectively, namely such that  $F(t,x,x_{t,\ep}) - \ep \leq F(t,x, w) $ and $F(s,x,x_{s,\ep}) - \ep \leq F(s,x, w) $ for 
every $w\in X$. 
Let us assume that $\sfd(x,x_{t,\ep}) \geq D^+(x,t)-\ep$ and $
\sfd(x,x_{s,\ep})\leq D^-(x,s)+\ep$.  We have
\[
\begin{split}
Q_sf(x)-Q_tf(x)&\leq F(s,x,x_{t,\ep})-F(t,x,x_{t,\ep}) +\ep=\frac{\sfd^p(x,x_{t,\ep})}{p}\frac{t^{p-1}-s^{p-1}}{t^{p-1}s^{p-1}}+\ep,\\
Q_sf(x)-Q_tf(x)&\geq
F(s,x,x_{s,\ep})-F(t,x,x_{s,\ep})+\ep=\frac{\sfd^p(x,x_{s,\ep})}{p}\frac{t^{p-1}-s^{p-1}}{t^{p-1}s^{p-1}}+\ep,
\end{split}
\]
For $\ep$ small enough, dividing by $s-t$ and using the definition of $x_{t,\ep}$ and $x_{s,\ep}$ we obtain
\[
\begin{split}
\frac{Q_sf(x)-Q_tf(x)}{s-t}&\leq \frac{( D^+(x,t)-\ep)^p }{qs^p}+\ep,\\
\frac{Q_sf(x)-Q_tf(x)}{s-t}&\geq
\frac{\sfd^p(x,x_{s,\ep})}{qt^{p}}+\ep,
\end{split}
\]
which gives as $\ep\to 0$ that $t\mapsto Q_tf(x)$ is Lipschitz in $(\delta,\infty)$
for any $\delta>0$ uniformly with respect to $x\in X$. Also, taking Proposition~\ref{prop:dmon} into account, we
get \eqref{eq:Dini1}. Now notice that from \eqref{eq:boundD} we get
that $q|\frac{\d}{\d t}Q_tf(x)|\leq p^q[\Lip(f)]^q$ for any $x\in X$
and a.e. $t$, which, together with the pointwise convergence of
$Q_tf$ to $f$ as $t\downarrow 0$, yields that $t\mapsto Q_tf\in
C(X)$ is Lipschitz in $[0,\infty)$.
\end{proof}

We will bound from above the slope of $Q_tf$ at $x$ with
$|D^+(x,t)/t|^{p-1}$; actually we shall prove a more precise statement,
which involves the \emph{asymptotic Lipschitz constant}
\begin{equation}\label{defn:asymlip}
{\rm Lip}_a(f,x):=\inf_{r>0}{\rm Lip}\bigl(f, B(x,r)\bigr)=\lim_{r\downarrow 0}{\rm Lip}\bigl(f, B(x,r)\bigr).
\end{equation}
We collect some properties of the asymptotic Lipschitz constant in the next proposition.

\begin{proposition}\label{prop:asymlip} Let $f:X\to\R$ be a Lipschitz function. Then
\begin{equation}\label{eq:asymlip}
{\rm Lip}(f)\geq {\rm Lip}_a(f,x)\geq |\nabla f|^*(x),
\end{equation}
where $|\nabla f|^*$ is the upper semicontinuous envelope 
of the slope of $f$. In length spaces the second inequality is an equality.
\end{proposition}
\begin{proof} The first inequality in \eqref{eq:asymlip} is trivial, while the second one follows
by the fact that ${\rm Lip}_a(f,\cdot)$ is upper semicontinuous and larger than $|\nabla f|$.
Since $|\nabla f|$ is an upper gradient of $f$, we have the inequality
$$
|f(y)-f(z)|\leq\int_0^{\ell(\gamma)}|\nabla f|(\gamma_t)\,dt
$$
for any curve $\gamma$ with constant speed joining $y$ to $z$. If $(X,\sfd)$ is a length space we can 
minimize w.r.t. $\gamma$ to get
$$
{\rm Lip}\bigl(f, B(x,r))\leq \sup_{B_{2r}(x)}|\nabla f|\leq\sup_{B_{2r}(x)}|\nabla f|^* .
$$
As $r\downarrow 0$ the inequality ${\rm Lip}_a(f,x)\leq|\nabla f|^*(x)$ follows.
\end{proof}

\begin{proposition}[Bound on the asymptotic Lipschitz constant of $Q_tf$]\label{prop:slopesqt}
For $(x,t)\in X\times (0,\infty)$ it holds: 
\begin{equation}
\label{eq:hjbss} {\rm Lip}_a(Q_tf,x)\leq
\biggl[\frac{D^+(x,t)}t\biggr]^{p-1}.
\end{equation}
\end{proposition}
In particular ${\rm Lip}(Q_t(f))\leq p{\rm Lip}(f)$.
\begin{proof}  
 Fix $y,\,z\in X$ and $t\in (0,\infty)$. For every $\ep>0$ let $y_\ep \in X$ be such that $F(t,y, y_\ep) - \ep \leq F(t,y, w) $ for 
every $w\in X$ and $|\sfd(y,y_{\ep}) -D^+(y,t)|\leq \ep$. Since it holds
\[
\begin{split}
Q_tf(z)-Q_tf(y)&\leq F(t,z, y_\ep)-F(t,y, y_\ep) +\ep=
f( y_\ep)+\frac{\sfd^p(z, y_\ep)}{pt^{p-1}}-f( y_\ep)-\frac{\sfd^p(y, y_\ep)}{pt^{p-1}} +\ep\\
&\leq
\frac{(\sfd(z,y)+\sfd(y, y_\ep))^p}{pt^{p-1}}-\frac{\sfd^p(y,y_\ep)}{pt^{p-1}}
 +\ep \\
&\leq \frac{\sfd(z,y)}{t^{p-1}}\bigl(\sfd(z,y)+D^+(y,t)+\ep\bigr)^{p-1} +\ep,
\end{split}
\]
so that letting $\ep \to 0$, dividing by $\sfd(z,y)$ and inverting the roles of $y$ and $z$ gives
$$
{\rm Lip}\bigl(Q_tf, B(x,r)\bigr)\leq t^{1-p}\bigl(\sup_{y\in  B(x,r)}D^+(y,t)\bigr)^{p-1}.
$$
Letting $r\downarrow 0$ and using the upper
semicontinuity of $D^+$ we get \eqref{eq:hjbss}.

Finally, the bound on the Lipschitz constant of $Q_tf$ follows
directly from \eqref{eq:boundD} and \eqref{eq:hjbss}.
\end{proof}

\begin{theorem}[Subsolution of HJ]\label{thm:subsol}
For every $x\in X$ it holds
\begin{equation}\label{eq:hjbsusbis}
\frac{\d}{\d t}Q_tf(x)+\frac{1}{q}{\rm Lip}_a^q(Q_tf,x)
\leq 0
\end{equation}
for every $t\in (0,\infty)$, with at most countably many exceptions.
\end{theorem}
\begin{proof}
The claim is a direct consequence of
Propositions~\ref{prop:timederivative} and
Proposition~\ref{prop:slopesqt}.
\end{proof}

Notice that \eqref{eq:hjbsusbis} is a stronger formulation of the HJ sub solution property
\begin{equation}\label{eq:hjbsus}
\frac{\d}{\d t}Q_tf(x)+\frac{1}{q}|\nabla f|^q(x)\leq 0,
\end{equation}
with the asymptotic Lipschitz constant ${\rm Lip}_a(Q_tf,\cdot)$ in place of
$|\nabla Q_t f|$.

\section{Weak gradients}\label{sec:weakgra}

Let $(X,\sfd)$ be a complete and separable metric space and let
$\mm$ be a nonnegative Borel measure in $X$ (not even $\sigma$-finiteness is needed
for the results of this section). In this
section we introduce and compare two notions of weak gradient, one obtained
by relaxation of the asymptotic Lipschitz constant, the other one obtained by
a suitable weak upper gradient property. Eventually we will show that the two
notions of gradient coincide: this will lead also to the coincidence with the
other intermediate notions of gradient considered in \cite{Cheeger00},
\cite{Koskela-MacManus}, \cite{Shanmugalingam00}, described in the appendix.

\subsection{Relaxed slope $|\nabla f|_{*,q}$}

The following definition is a variation of the one considered in \cite{Cheeger00} 
(where the relaxation procedure involved upper gradients) and of the one considered
in \cite{Ambrosio-Gigli-Savare11} (where the relaxation procedure involved slopes
of Lipschitz functions). The use of the (stronger) asymptotic Lipschitz constant has beed
suggested in the final section of \cite{Ambrosio-Gigli-Savare12}: it
is justified by the subsolution property \eqref{eq:hjbsusbis} and it leads to stronger density results.
In the spirit of the Sobolev space theory, these 
should be considered as ``$H$ definitions'', since 
approximation with Lipschitz functions with bounded support are involved.

\begin{definition}[Relaxed slope]\label{def:genuppergrad} We say that $g\in L^q(X,\mm)$ is a
$q$-relaxed slope of $f\in L^q(X,\mm)$ if there exist $\tilde{g}\in
L^q(X,\mm)$ and Lipschitz functions with bounded support $f_n$ such that:
\begin{itemize}
\item[(a)] $f_n\to f$ in $L^q(X,\mm)$ and ${\rm Lip}_a(f_n,\cdot)$ weakly converge to
$\tilde{g}$ in $L^q(X,\mm)$;
\item[(b)] $\tilde{g}\leq g$ $\mm$-a.e. in $X$.
\end{itemize}
We say that $g$ is the minimal $q$-relaxed slope of $f$ if its
$L^q(X,\mm)$ norm is minimal among $q$-relaxed slopes. We shall
denote by $\relgradq fq$ the minimal $q$-relaxed slope.
\end{definition}

By this definition and the sequential compactness of weak
topologies, any $L^q$ limit of Lipschitz functions $f_n$ with
$\int {\rm Lip}_a^q(f_n,\cdot)\,\d\mm$ uniformly bounded has a $q$-relaxed
slope. On the other hand, using Mazur's lemma (see
\cite[Lemma~4.3]{Ambrosio-Gigli-Savare11} for details), the
definition of $q$-relaxed slope would be unchanged if the weak
convergence of ${\rm Lip}_a(f_n,\cdot)$ in (a) were replaced by the condition
${\rm Lip}_a(f_n,\cdot)\leq g_n$ and $g_n\to\tilde{g}$ strongly in
$L^q(X,\mm)$. This alternative characterization of $q$-relaxed
slopes is suitable for diagonal arguments and proves, together with
\eqref{eq:subadd}, that the collection of $q$-relaxed slopes is a
closed convex set, possibly empty. Hence, thanks to the uniform
convexity of $L^q(X,\mm)$, the definition of $\relgradq fq$ is well
posed. Also, arguing as in \cite{Ambrosio-Gigli-Savare11} and using once more the uniform
convexity of $L^q(X,\mm)$, it is not difficult to show the following result:

\begin{proposition}\label{prop:easy}
If $f\in L^q(X,\mm)$ has a $q$-relaxed slope then there exist
Lipschitz functions $f_n$ with bounded support satisfying
\begin{equation}\label{densitylip1}
\lim_{n\to\infty}\int_X|f_n-f|^q\,\d\mm+\int_X\bigl|{\rm Lip}_a(f_n,\cdot)-|\nabla f|_{*,q}\bigr|^q\,\d\mm=0.
\end{equation}
\end{proposition}

Notice that in principle the integrability of $f$ could be
decoupled from the integrability of the gradient, because no global
Poincar\'e inequality can be expected at this level of generality.
Indeed, to increase the symmetry with the definition of weak upper gradient
(which involves no integrability assumption on $f$), one
might even consider the convergence $\mm$-a.e. of the approximating
functions, removing any integrability assumption. We have left the
convergence in $L^q$ because this presentation is more consistent
with the usual presentations of Sobolev spaces, and the definitions
given in \cite{Cheeger00} and \cite{Ambrosio-Gigli-Savare11}. Using
locality and a truncation argument, the definitions can be extended
to more general classes of functions, see
\eqref{eq:extendedrelaxed}.

\begin{lemma}[Pointwise minimality of $\relgradq fq$]\label{le:local}
Let $g_1,\,g_2$ be two $q$-relaxed slopes of $f$. Then
$\min\{g_1,g_2\}$ is a $q$-relaxed slope as well. In particular, not
only the $L^q$ norm of $\relgradq fq$ is minimal, but also
$\relgradq fq\leq g$ $\mm$-a.e. in $X$ for any relaxed slope $g$ of
$f$.
\end{lemma} 
\begin{proof} We argue as in \cite{Cheeger00}, \cite{Ambrosio-Gigli-Savare11}.
First we notice that for every $f,\,g \in \Lip(X)$
\begin{equation}
\label{eq:subadd-relgradq}\Lip_a(f+g,x) \leq \Lip_a(f,x) +  \Lip_a(g,x)  \qquad \forall x\in X,
\end{equation}
\begin{equation}
\label{eq:leibn-relgradq}\Lip_a(fg,x) \leq |f(x)|  \Lip_a(g,x) +  |g(x)| \Lip_a(f,x) \qquad \forall x\in X.
\end{equation}
Indeed \eqref{eq:subadd-relgradq} is obvious; for \eqref{eq:leibn-relgradq} we have that
$$|f(z)g(z)-f(y)g(y)| \leq |f(z)||g(z)-g(y)|+|g(y)| |f(z)-f(y)| \qquad \forall y,\,z \in X,$$
so that
$$\Lip (fg,B(x,r)) \leq \sup_{z\in B(x,r)}|f(z)|  \Lip(g,B(x,r)) +  \sup_{y\in B(x,r)}|g(y)| \Lip(f,B(x,r)) \qquad \forall x\in X$$
 and we let $r\to 0$.

It is sufficient to prove that if $B\subset X$ is a Borel set, then
$\nchi_Bg_1+\nchi_{X\setminus B}g_2$ is a $q$-relaxed slope of $f$. By
approximation, taking into account the closure of the class of
$q$-relaxed slopes, we can assume with no loss of generality that $B$ is
an open set. We fix $r>0$ and a Lipschitz function $\phi_r:X\to
[0,1]$ equal to $0$ on $X\setminus B_r$ and equal to $1$ on
$B_{2r}$, where the open sets $B_s\subset B$ are defined by
$$
B_s:=\left\{x\in X:\ {\rm dist}(x,X\setminus B)> s\right\}\subset B.
$$
Let now $f_{n,i}$, $i=1,\,2$, be Lipschitz functions with bounded support
converging to $f$ in $L^q(X,\mm)$ as $n\to\infty$, with ${\rm Lip}_a( {f_{n,i}, \cdot})$ weakly convergent to $\tilde g_i$ in $L^q(X,\mm)$ and set $f_n:=\phi_r
f_{n,1}+(1-\phi_r)f_{n,2}$. Then, ${\rm Lip}_a( {f_{n}, \cdot})= {\rm Lip}_a( {f_{n,1}, \cdot})$ on
$B_{2r}$ and ${\rm Lip}_a( {f_{n}, \cdot})= {\rm Lip}_a( {f_{n,2}, \cdot})$ on
$X\setminus\overline{B_r}$; for every $x\in \overline{B_r}\setminus B_{2r}$, by
applying \eqref{eq:subadd-relgradq} to $f_{n,2}$ and $\phi_r
(f_{n,1}-f_{n,2})$ and by applying \eqref{eq:leibn-relgradq} to $\phi_r$ and$(f_{n,1}-f_{n,2})$ , we can estimate
$$
{\rm Lip}_a( {f_{n}, x}) \leq {\rm Lip}_a( {f_{n,2}, x})+{\rm
Lip}(\phi_r)|f_{n,1}(x)-f_{n,2}(x)|+ \phi_r\bigl( {\rm Lip}_a( {f_{n,1}, x})+ {\rm Lip}_a( {f_{n,2}, x})\bigr).
$$
Since $\overline{B_r}\subset B$, by taking weak limits of a
subsequence, it follows that
$$
\nchi_{B_{2r}}g_1+\nchi_{X\setminus\overline{B_r}}g_2+\nchi_{B\setminus
B_{2r}}(g_1+2g_2)
$$
is a $q$-relaxed slope of $f$. Letting $r\downarrow 0$ gives that
$\nchi_Bg_1+\nchi_{X\setminus B}g_2$ is a $q$-relaxed slope as well.

For the second part of the statement argue by contradiction: let $g$
be a $q$-relaxed slope of $f$ and assume that $B=\{g<\relgradq fq\}$ is
such that $\mm(B)>0$. Consider the $q$-relaxed slope $g\nchi_B+\relgradq
fq \nchi_{X\setminus B}$: its $L^q$ norm is strictly less than the
$L^q$ norm of $\relgradq fq$, which is a contradiction.
\end{proof}

The previous pointwise minimality property immediately yields
\begin{equation}
\label{eq:facile} \relgradq  fq\leq {\rm Lip}_a(f,\cdot)\qquad\text{$\mm$-a.e.
in $X$}
\end{equation}
for any Lipschitz function $f:X\to\R$ with bounded support. Since both objects are
local, the inequality immediately extends by a truncation argument to all functions
$f\in L^q(X,\mm)$ with a $q$-relaxed slope, Lipschitz on bounded sets.

Also the proof of locality and chain rule is quite standard, see
\cite{Cheeger00} and \cite[Proposition~4.8]{Ambrosio-Gigli-Savare11}
for the case $q=2$ (the same proof works in the general case).

\begin{proposition}[Locality and chain rule]\label{prop:chain}
If $f\in L^q(X,\mm)$ has a $q$-relaxed slope, the following
properties hold.
\begin{itemize}
\item[(a)] $\relgradq hq=\relgradq fq$ $\mm$-a.e. in $\{h=f\}$
whenever $f$ has a $q$-relaxed slope.
\item[(b)] $\relgradq {\phi(f)}q\leq |\phi'(f)|\relgradq fq$ for any $C^1$ and Lipschitz function
$\phi$ on an interval containing the image of $f$. Equality holds if
$\phi$ is nondecreasing.
\end{itemize}
\end{proposition}

\subsection{$q$-weak upper gradients and $|\nabla f|_{w,q}$}

Recall that the evaluation maps $\rme_t:C([0,1],X)\to X$ are defined
by $\rme_t(\gamma):=\gamma_t$. We also introduce the restriction
maps ${\rm restr}_t^s: C([0,1],X)\to C([0,1],X)$, $0\le t\le s\le
1$, given by
\begin{equation}
{\rm restr}_t^s(\gamma)_r:=\gamma_{(1-r)t+rs},\label{eq:93}
\end{equation}
so that ${\rm restr}_t^s$ ``stretches'' the restriction of the curve
to $[s,t]$ to the whole of $[0,1]$.

Our definition of $q$-weak upper gradient is inspired by
\cite{Koskela-MacManus}, \cite{Shanmugalingam00}, 
allowing for exceptional curves in \eqref{eq:uppergradient}, but with a different notion
of exceptional set, compared to \cite{Koskela-MacManus}, \cite{Shanmugalingam00}.

\begin{definition}[Test plans and negligible sets of curves]\label{def:testplans}
We say that a probability measure $\ppi\in\prob{C([0,1],X)}$ is a
$p$-\emph{test plan} if $\ppi$ is concentrated on $AC^p([0,1],X)$,
$\iint_0^1|\dot\gamma_t|^p\d t\,\d\ppi<\infty$ and there exists a
constant $C(\ppi)$ such that
\begin{equation}
(\e_t)_ \sharp\ppi \leq C(\ppi)\mm\qquad\forall t\in[0,1].
\label{eq:1}
\end{equation}
A set $A\subset C([0,1],X)$ is said to be
$q$-\emph{negligible} if it is contained in a $\ppi$-negligible set for any $p$-test plan $\ppi$. 
A property which holds for every $\gamma\in C([0,1],X)$, except
possibly a $q$-negligible set, is said to hold for $q$-almost every
curve.
\end{definition}
Observe that, by definition, $C([0,1],X)\setminus AC^p([0,1],X)$ is
$q$-negligible, so the notion starts to be meaningful when we look
at subsets of $AC^p([0,1],X)$. 
\begin{remark}
  \label{re:easy}
  \upshape
  An easy consequence of condition \eqref{eq:1} is that if two
  $\mm$-measurable functions $f,\,g:X\to\R$ coincide up to a
  $\mm$-negligible set and $\mathcal T$ is an at most countable subset
  of $[0,1]$, then the functions
  $f\circ \gamma$ and $g\circ \gamma$ coincide in $\mathcal T$ 
  for $q$-almost every curve
  $\gamma$. 

  Moreover, choosing an arbitrary $p$-test plan $\ppi$ and applying Fubini's
  Theorem to the product measure $\Leb 1\times \ppi$
  in $(0,1)\times C([0,1];X)$ we also obtain that
  $f\circ\gamma=g\circ\gamma$ $\Leb 1$-a.e.\ in $(0,1)$ for
  $\ppi$-a.e.\ curve $\gamma$; since $\ppi$ is arbitrary, the same
  property holds for $q$-a.e.\ $\gamma$.
\end{remark}

Coupled with the definition of
$q$-negligible set of curves, there are the definitions of $q$-weak upper gradient and
 of functions which are Sobolev along $q$-a.e. curve.
\begin{definition}[$q$-weak upper gradients]
A Borel function
$g:X\to[0,\infty]$ is a $q$-weak upper gradient of $f:X\to \R$ if
\begin{equation}
\label{eq:inweak} \left|\int_{\partial\gamma}f\right|\leq
\int_\gamma g<\infty\qquad\text{for $q$-a.e. $\gamma$.}
\end{equation}
\end{definition}

\begin{definition}[Sobolev functions along $q$-a.e. curve]
 A function $f:X\to\R$ is Sobolev along $q$-a.e. curve if for
$q$-a.e. curve $\gamma$ the function $f\circ\gamma$ coincides a.e.
in $[0,1]$ and in $\{0,1\}$ with an absolutely continuous map
$f_\gamma:[0,1]\to\R$.
\end{definition}

By Remark \ref{re:easy} applied to $\mathcal T:=\{0,1\}$, \eqref{eq:inweak} does not depend on
the particular representative of $f$ in the class of $\mm$-measurable
function coinciding with $f$ up to a $\mm$-negligible set. 
The same Remark also shows that  the property of being Sobolev along
$q$-q.e.\ curve $\gamma$ is independent of the representative in the
class of $\mm$-measurable functions coinciding with $f$ $\mm$-a.e.\
in $X$.

In the next proposition, based on Lemma~\ref{lem:Fibonacci}, we prove that the existence of a $q$-weak
upper gradient $g$ implies Sobolev regularity along $q$-a.e.\ curve. 

\begin{proposition}
  \label{prop:restr}
  Let $f:X\to\R$ be $\mm$-measurable, and let $g$ be a $q$-weak upper gradient. Then $f$ is Sobolev along $q$-a.e. curve.
\end{proposition}
\begin{proof}
  Notice that if $\ppi$ is a $p$-test plan, so is $({\rm
restr}_t^s)_\sharp\ppi$. Hence if $g$ is a $q$-weak upper gradient
of $f$ such that $\int_\gamma g<\infty$ for
$q$-a.e.\ $\gamma$, then for every $t<s$ in $[0,1]$ it holds
    \[
    |f(\gamma_s)-f(\gamma_t)|\leq \int_t^s g(\gamma_r)|\dot\gamma_r|\,\d
    r \qquad\text{for $q$-a.e. $\gamma$.}
    \]
    Let $\ppi$ be a $p$-test plan: by Fubini's theorem applied
    to the product measure $\Leb2\times\ppi$ in $(0,1)^2\times
    C([0,1];X)$, it follows that for $\ppi$-a.e. $\gamma$ the function
     $f$ satisfies
    \[
    |f(\gamma_s)-f(\gamma_t)|\leq \Bigl|\int_t^s g(\gamma_r)|\dot\gamma_r|\,\d
    r \Bigr|\qquad\text{for $\Leb{2}$-a.e. $(t,s)\in (0,1)^2$.}
    \]
    An analogous argument shows that 
    \begin{equation}
      \label{eq:2}
      \left\{
    \begin{aligned}
      \textstyle |f(\gamma_s)-f(\gamma_0)|&\textstyle 
      \leq \int_0^s
      g(\gamma_r)|\dot\gamma_r|\,\d r\\
      \textstyle |f(\gamma_1)-f(\gamma_s)|&\textstyle \leq \int_s^1
      g(\gamma_r)|\dot\gamma_r|\,\d r
    \end{aligned}\right.
    \qquad\text{for $\Leb{1}$-a.e. $s\in (0,1)$.}
\end{equation}
 Since $g\circ \gamma|\dot \gamma|\in L^1(0,1)$ for
    $\ppi$-a.e.\ $\gamma$,  
    by Lemma~\ref{lem:Fibonacci} it follows that $f\circ\gamma\in W^{1,1}(0,1)$
    for $\ppi$-a.e. $\gamma$, and
    \begin{equation}\label{eq:pointwisewug}
      \biggl|\frac{\d}{\dt}(f\circ\gamma)\biggr|\leq
      g\circ\gamma|\dot\gamma|\quad\text{a.e. in $(0,1)$, for
        $\ppi$-a.e. $\gamma$.}
    \end{equation}
  Since $\ppi$ is arbitrary, we conclude that $f\circ\gamma\in
  W^{1,1}(0,1)$ for $q$-a.e.\ $\gamma$, and therefore it admits an
  absolutely continuous representative $f_\gamma$; moreover,
  by \eqref{eq:2}, it is immediate to check that $f(\gamma(t))=f_\gamma(t)$ for $t\in \{0,1\}$ and $q$-a.e.\ $\gamma$.
\end{proof}

Using the same argument given in the previous proposition it is
immediate to show that 
\begin{equation}\label{eq:locweak}
\text{$g_i$, $i=1,2$ $q$-weak upper gradients of $f$}\quad\Longrightarrow\quad \text{$\min\{g_1,g_2\}$ $q$-weak upper gradient of $f$.}
\end{equation}
Using this stability property we can recover, as we did for relaxed slopes, a distinguished
minimal object.

\begin{definition}[Minimal $q$-weak upper gradient]
  Let $f:X\to\R$ be a $\mm$-measurable function having a $q$-weak upper gradient.
  The minimal $q$-weak upper gradient $\weakgradq fq$ of $f$
  is the $q$-weak upper gradient characterized, up to
$\mm$-negligible sets, by the property
\begin{equation}\label{eq:defweakgrad}
  \weakgradq fq\leq g\qquad\text{$\mm$-a.e. in $X$, for every $q$-weak upper
    gradient $g$ of $f$.}
\end{equation}
\end{definition}

Uniqueness of the minimal weak upper gradient is obvious. For
existence, since $\mm$ is $\sigma$-finite we can find a Borel and
$\mm$-integrable function $\theta:X\to (0,\infty)$ and $\weakgradq
fq :=\inf_n g_n$, where $g_n$ are $q$-weak upper gradients which
provide a minimizing sequence in
$$
\inf\left\{\int_X \theta\, {\rm tan}^{-1}g\,\d\mm:\ \text{$g$ is a
$q$-weak upper gradient of $f$}\right\}.
$$
We immediately see, thanks to \eqref{eq:locweak}, that we can assume
with no loss of generality that $g_{n+1}\leq g_n$. Hence, by
monotone convergence, the function $\weakgradq fq$ is a $q$-weak
upper gradient of $f$ and $\int_X \theta\,{\rm tan}^{-1}g\,\d\mm$ is
minimal at $g=\weakgradq fq$. This minimality, in conjunction with
\eqref{eq:locweak}, gives \eqref{eq:defweakgrad}.

Next we consider the stability of $q$-weak upper gradients (analogous to the stability result
given in \cite[Lemma~4.11]{Shanmugalingam00}). We shall actually need a slightly more general
statement, which involves a weaker version of the upper gradient property (when $\varepsilon=0$
we recover the previous definition, since curves with 0 length are constant).

\begin{definition}[$q$-weak upper gradient up to scale $\ep$]\label{dfn:wug-discr} Let $f:X\to\R$. We say that a Borel
function $g:X\to [0,\infty)$ is a 
$q$-weak upper gradient of $f$ up to scale $\varepsilon\geq 0$ if  
for $q$-a.e. curve $\gamma\in AC^p([0,1];X)$ such that
$$ \varepsilon<\int_0^1 |\dot{\gamma}_t|\, \d t$$
it holds
\begin{equation}\label{eqn:grad-discr}
\biggl| \int_{\partial\gamma} f  \biggr|\leq \int_\gamma g<\infty .
\end{equation}
\end{definition}

\begin{theorem}[Stability w.r.t. $\mm$-a.e. convergence]\label{thm:stabweak}
Assume that $f_n$ are $\mm$-measurable, $\varepsilon_n\geq 0$ 
and that $g_n\in L^q(X,\mm)$ are $q$-weak upper gradients of $f_n$ up to scale $\varepsilon_n$. Assume
furthermore that $f_n(x)\to f(x)\in\R$ for $\mm$-a.e. $x\in X$, $\varepsilon_n\to\varepsilon$ and
that $(g_n)$ weakly converges to $g$ in $L^q(X,\mm)$. Then $g$ is a
$q$-weak upper gradient of $f$ up to scale $\varepsilon$.
\end{theorem}
\begin{proof}
Fix a $p$-test plan $\ppi$. 
We have to show that \eqref{eqn:grad-discr}
holds for $\ppi$-a.e. $\gamma$ with $\int_0^1|\dot\gamma_t|\,\d t>\varepsilon$.
Possibly restricting $\ppi$ to a smaller set of curves, we can assume with no
loss of generality that
$$
\int_0^1|\dot\gamma_t|\,\d t>\varepsilon'\qquad\text{for $\ppi$-a.e. $\gamma$}
$$
for some $\varepsilon'>\varepsilon$. We consider in the sequel integers $h$ sufficiently large, 
such that $\varepsilon_h\leq\varepsilon'$.

By Mazur's theorem we can find convex
combinations
$$
h_n:=\sum_{i=N_h+1}^{N_{h+1}}\alpha_ig_i\qquad\text{with
$\alpha_i\geq 0$, $\sum_{i=N_h+1}^{N_{h+1}}\alpha_i=1$,
$N_h\to\infty$}
$$
converging strongly to $g$ in $L^q(X,\mm)$. Denoting by $\tilde f_n$
the corresponding convex combinations of $f_n$, $h_n$ are $q$-weak upper
gradients of $\tilde f_n$ and still $\tilde f_n\to f$ $\mm$-a.e. in
$X$.

Since for every nonnegative Borel function $\varphi:X\to [0,\infty]$
it holds (with $C=C(\ppi)$)
\begin{align}
  \notag\int\Big(\int_{\gamma}\varphi\Big)\,\d\ppi&=
  \int\Big(\int_0^1 \varphi(\gamma_t)|\dot
  \gamma_t|\,\d t\Big)\,\d\ppi
  \le
  \int\Big(\int_0^1\varphi^q(\gamma_t)\,\d
  t\Big)^{1/q}
  \Big(\int_0^1 |\dot
  \gamma_t|^p\,\d t\Big)^{1/p}\,\d\ppi
  \\&\notag
  \le
  \Big(\int_0^1 \int\varphi^q\,\d(\rme_t)_\sharp\ppi\,\d
  t\Big)^{1/q}
  \Big(\iint_0^1|\dot\gamma_t|^p\,\d t\,\d\ppi\Big)^{1/p}
\\ &\le  \Big(C\int\varphi^q\,\d\mm\Big)^{1/q}  \Big(\iint_0^1|\dot\gamma_t|^p\,\d t\,\d\ppi\Big)^{1/p},
\label{eq:21}
 \end{align}
we obtain
$$
\int\int_{\gamma}|h_n-g|
\,\d\ppi\leq C^{1/q}\bigl(\iint_0^1|\dot\gamma_t|^p\,\d
t\,\d\ppi\bigr)^{1/p}
\|h_n-g\|_q
\to 0.
$$
Hence we can find a subsequence $n(k)$  such that
$$\lim_{k\to\infty}\int_\gamma|h_{n(k)}-g|\to 0\qquad\text{
for $\ppi$-a.e. $\gamma$.}$$ 
Since $\tilde{f}_n$ converge
$\mm$-a.e. to $f$ and the marginals of $\ppi$ are absolutely
continuous w.r.t. $\mm$ we have also that for $\ppi$-a.e. $\gamma$
it holds $\tilde{f}_n(\gamma_0)\to f(\gamma_0)$ and
$\tilde{f}_n(\gamma_1)\to f(\gamma_1)$.

If we fix a curve $\gamma$ satisfying these convergence properties,
we can pass to the limit as $k\to\infty$ in the inequalities
$| \int_{\partial\gamma} \tilde{f}_{n(k)}  | \leq \int_\gamma h_{n(k)}$ to get
$| \int_{\partial\gamma} f  | \leq \int_\gamma g$.
\end{proof}

Combining Proposition~\ref{prop:easy} with the fact that the asymptotic Lipschitz
constant is an upper gradient (and in particular a $q$-weak upper gradient), the
previous stability property gives that $|\nabla f|_{*,q}$ is a $q$-weak upper gradient.
Then, \eqref{eq:defweakgrad} gives
\begin{equation}\label{allinequalities1}
\weakgradq fq\leq |\nabla f|_{*,q}\qquad\text{$\mm$-a.e. in $X$}
\end{equation}
whenever $f\in L^q(X,\mm)$ has a $q$-relaxed slope. The proof of the converse inequality
(under no extra assumption on the metric measure structure)
requires much deeper ideas, described in the next two sections.

\section{Gradient flow of ${\bf C}_q$ and energy dissipation}\label{sec:Chq}

In this section we assume that $(X,\sfd)$ is complete and separable, and that
$\mm$ is a finite Borel measure.

As in the previous sections, $q\in (1,\infty)$ and $p$ is the dual exponent. In order
to apply the theory of gradient flows of convex functionals in
Hilbert spaces, when $q>2$ we need to extend $\relgradq fq$ also to
functions in $L^2(X,\mm)$ (because Definition~\ref{def:genuppergrad}
was given for $L^q(X,\mm)$ functions). To this aim, we denote
$f_N:=\max\{-N,\min\{f,N\}\}$ and set
\begin{equation}\label{eq:mathcalC}
\mathcal C:=\left\{f:X\to\R:\ \text{$f_N$ has a $q$-relaxed slope
for all $N\in\N$}\right\}.
\end{equation}
Accordingly, for all $f\in\mathcal C$ we set
\begin{equation}\label{eq:extendedrelaxed}
\relgradq fq:=\relgradq {f_N}q\qquad\text{$\mm$-a.e. in $\{|f|<N\}$}
\end{equation}
for all $N\in\N$. We can use the locality property in
Proposition~\ref{prop:chain}(a) to show that this definition is well
posed, up to $\mm$-negligible sets, and consistent with the previous
one. Furthermore, locality and chain rules still apply, so we shall
not use a distinguished notation for the new gradient.

We define an auxiliary functional, suitable for the Hilbertian energy dissipation estimates, by
\begin{equation}\label{def:Cheeger}
{\bf C}_q(f):=\frac{1}{q}\int_X |\nabla f|_{*,q}^q \,\d\mm,
\end{equation}
set to $+\infty$ if $f\in L^2(X,\mm)\setminus\mathcal C$.

\begin{theorem} \label{thm:cheeger} The functional
${\bf C}_q$ is convex and lower semicontinuous in $L^2(X,\mm)$.
\end{theorem}
\begin{proof} The proof of convexity is elementary, so we focus on
lower semicontinuity. Let $(f_n)$ be convergent to $f$ in
$L^2(X,\mm)$ and assume, possibly extracting a subsequence
and with no loss of generality, that ${\bf C}_q(f_n)$ converges to a
finite limit.

Assume first that all $f_n$ are uniformly bounded, so that $f_n\to f$ also in
$L^q(X,\mm)$ (because $\mm$ is finite). Let
$f_{n(k)}$ be a subsequence such that $\relgradq {f_{n(k)}}q$ weakly
converges to $g$ in $L^q(X,\mm)$. Then $g$ is a $q$-relaxed slope of
$f$ and
$$
{\bf C}_q(f)\leq\frac1q\int_X|g|^q\,\d\mm\leq\liminf_{k\to\infty}\frac1q
\int_X|\nabla f_{n(k)}|^q_{*,q}\,\d\mm
=\liminf_{n\to\infty}{\bf C}_q(f_n).
$$
In the general case when $f_n\in{\mathcal C}$ we consider the
functions $f^N_n:=\max\{-N,\min\{f,N\}\}$ to conclude from the
inequality $|\nabla f^N_n|_{*,q}\leq|\nabla f_n|_{*,q}$ that
$f^N:=\max\{-N,\min\{f,N\}\}$ has $q$-relaxed slope for any $N\in\N$
and
$$
\int_X|\nabla f^N|_{*,q}^q\,\d\mm\leq \liminf_{n\to\infty}
\int_X|\nabla f^N_n|_{*,q}^q\,\d\mm\leq\liminf_{n\to\infty}
\int_X|\nabla f_n|_{*,q}^q\,\d\mm.
$$
Passing to the limit as $N\to\infty$, the conclusion follows by
monotone convergence.
\end{proof}

\begin{remark}\label{rem:basiclsc} {\rm More generally, the same argument proves the
$L^2(X,\mm)$-lower semicontinuity of the functional
$$
f\mapsto\int_X \frac{|\nabla f|_{*,q}^q}{|f|^\alpha}\,\d\mm
$$
in $\mathcal C$, for any $\alpha>0$. Indeed, locality and chain rule
allow the reduction to nonnegative functions $f_n$ and we can use
the truncation argument of Theorem~\ref{thm:cheeger} to reduce
ourselves to functions with values in an interval $[c,C]$ with
$0<c\leq C<\infty$. In this class, we can again use the chain rule
to prove the identity
$$
\int_X|\nabla f^\beta|^q_{*,q}\,\d\mm= |\beta|^q\int_X\frac{|\nabla
f|_{*,q}^q}{|f|^\alpha}\,\d\mm
$$
with $\beta:=1-\alpha/q$ to obtain the result when $\alpha\neq q$.
If $\alpha=q$ we use a logarithmic transformation.
 }\fr
\end{remark}

Since the finiteness domain of ${\bf C}_q$ is dense in $L^2(X,\mm)$ (it
includes bounded Lipschitz functions), the Hilbertian theory of
gradient flows (see for instance \cite{Brezis73},
\cite{Ambrosio-Gigli-Savare08}) can be applied to Cheeger's
functional \eqref{def:Cheeger} to provide, for all $f_0\in
L^2(X,\mm)$, a locally absolutely continuous map $t\mapsto f_t$ from
$(0,\infty)$ to $L^2(X,\mm)$, with $f_t\to f_0$ as $t\downarrow 0$,
whose derivative satisfies
\begin{equation}\label{eq:ODE}
\frac{d}{dt}f_t\in -\partial^- {\bf C}_q(f_t)\qquad\text{for a.e. $t\in
(0,\infty)$.}
\end{equation}

Having in mind the regularizing effect of gradient flows, namely the
selection of elements with minimal $L^2(X,\mm)$ norm in
$\partial^-{\bf C}_q$, the following definition is natural.

\begin{definition}[$q$-Laplacian]\label{def:delta}
The $q$-Laplacian $\Delta_q f$ of $f\in L^2(X,\mm)$ is defined for
those $f$ such that $\partial^- {\bf C}_q(f)\neq\emptyset$. For those $f$,
$-\Delta_q f$ is the element of minimal $L^2(X,\mm)$ norm in
$\partial^-{\bf C}_q(f)$. The domain of $\Delta_q$ will be denoted by
$D(\Delta_q)$.
\end{definition}

It should be observed that, even in the case $q=2$, in general the
Laplacian is \emph{not} a linear operator. For instance, if $X=\R^2$
endowed with the sup norm $\|(x,y)\|=\max\{|x|,|y|\}$, then
$$
{\bf C}_2(f)=\int_{\R^2}\biggl(\biggl|\frac{\partial f}{\partial x}\biggr|+\biggl|\frac{\partial f}{\partial y}\biggr|\biggr)^2
\,\d x \d y.
$$
Since ${\bf C}_2$ is not a quadratic form, its subdifferential is not linear.

Coming back to our general framework, the trivial implication
\[
v\in\partial^-{\bf C}_q(f)\qquad\Longrightarrow\qquad \lambda^{q-1}
v\in\partial^-{\bf C}_q(\lambda f),\quad\forall \lambda\in\R,
\]
still 
ensures that the $q$-Laplacian (and so the gradient flow of ${\bf C}_q$)
is $(q-1)$-homogenous. 

We can now write
$$
\frac{\d}{\d t}f_t=\Delta_q f_t
$$
for gradient flows $f_t$ of ${\bf C}_q$, the derivative being understood
in $L^2(X,\mm)$, in accordance with the classical case.

\begin{proposition}[Integration by parts]
\label{prop:deltaineq} For all $f\in D(\Delta_q)$, $g\in D({\bf C}_q)$ it
holds
\begin{equation}
\label{eq:delta1} -\int_X g\Delta_q f\,\d\mm\leq \int_X \relgradq
gq|\nabla f|_{*,q}^{q-1}\,\d\mm.
\end{equation}
Equality holds if $g=\phi(f)$ with $\phi\in C^1(\R)$ with bounded
derivative on the image of $f$.
\end{proposition}
\begin{proof}
Since $-\Delta_q f\in\partial^-{\bf C}_q(f)$ it holds
\[
{\bf C}_q(f)-\int_X \eps g\Delta_q f\,\d\mm\leq {\bf C}_q(f+\eps
g),\qquad\forall g\in L^q(X,\mm),\,\,\eps\in\R.
\]
For $\eps>0$, $\relgradq fq+\eps \relgradq gq$ is a $q$-relaxed
slope of $f+\eps g$ (possibly not minimal) whenever $f$ and $g$ have
$q$-relaxed slope. By truncation, it is immediate to obtain from
this fact that $f,\,g\in\mathcal C$ implies $f+\eps g\in\mathcal C$
and
$$
\relgradq {(f+\eps g)}q \leq\relgradq fq+\eps \relgradq
gq\qquad\text{$\mm$-a.e. in $X$.}
$$
Thus it holds $q{\bf C}_q(f+\eps g)\leq\int_X(\relgradq fq+\eps\relgradq
gq)^q\,\d\mm$ and therefore
\[
-\int_X\eps g\Delta_q f\,\d\mm\leq \frac1q\int_X(\relgradq
fq+\eps\relgradq gq)^q-|\nabla f|_{*,q}^q\,\d\mm=\eps\int_X\relgradq
gq|\nabla f|^{q-1}_{*,q}\,\d\mm+o(\eps).
\]
Dividing by $\eps$ and letting $\eps\downarrow 0$ we get
\eqref{eq:delta1}.

For the second statement we recall that $\relgradq {(f+\eps
\phi(f))}q=(1+\eps \phi'(f))\relgradq fq$ for $|\eps|$ small enough.
Hence
\[
{\bf C}_q(f+\eps \phi(f))-{\bf C}_q(f)= \frac{1}{q}\int_X|\nabla
f|_{*,q}^q\bigl((1+\eps \phi'(f))^q-1\bigr)\,\d\mm=\eps\int_X|\nabla
f|_{*,q}^q \phi'(f)\,\d\mm+o(\eps),
\]
which implies that for any $v\in \partial^-{\bf C}_q(f)$ it holds
$\int_Xv \phi(f)\,\d\mm=\int_X|\nabla f|_{*,q}^q\phi'(f)\,\d\mm$,
and gives the thesis with $v=-\Delta_q f$.
\end{proof}

\begin{proposition}[Some properties of the gradient flow of ${\bf C}_q$]\label{prop:basecal}
Let $f_0\in L^2(X,\mm)$ and let $(f_t)$ be the gradient flow of
${\bf C}_q$ starting from $f_0$. Then the following properties hold.\\*
\noindent (Mass preservation) $\int f_t\,\d\mm=\int f_0\,\d\mm$ for
any $t\geq 0$.\\* \noindent (Maximum principle) If $f_0\leq C$
(resp. $f_0\geq c$) $\mm$-a.e. in $X$, then $f_t\leq C$ (resp
$f_t\geq c$) $\mm$-a.e. in $X$ for any $t\geq 0$.\\* (Energy
dissipation) Suppose $0<c\leq f_0\leq C<\infty$ $\mm$-a.e. in $X$
and $\Phi\in C^2([c,C])$. Then $t\mapsto\int\Phi(f_t)\,\d\mm$ is
locally absolutely continuous in $(0,\infty)$ and it holds
\[
\frac{\d}{\dt}\int \Phi(f_t)\,\d\mm=-\int\Phi''(f_t)|\nabla
f_t|_{*,q}^q\,\d\mm\qquad\text{for a.e. $t\in (0,\infty)$.}
\]
\end{proposition}
\begin{proof} (Mass preservation) Just notice that from \eqref{eq:delta1} we get
\[
\left|\frac{\d}{\dt}\int f_t\,\d\mm\right|=\left|\int
\mathbf{1}\cdot\Delta_q f_t\,\d\mm\right|\leq\int\relgradq{\mathbf
1}q{|\nabla f_t|_{*,q}^q}\,\d\mm=0\quad\text{for a.e. $t>0$},
\]
where $\mathbf 1$ is the function identically equal to 1, which has
minimal $q$-relaxed slope equal to 0 by \eqref{eq:facile}.\\*
(Maximum principle) Fix $f\in L^2(X,\mm)$, $\tau>0$ and, according
to the so-called implicit Euler scheme, let $f^\tau$ be the unique
minimizer of
\[
g\qquad\mapsto\qquad {\bf C}_q(g)+\frac{1}{2\tau}\int_X|g-f|^2\,\d\mm.
\]
Assume that $f\leq C$. We claim that in this case $f^\tau\leq C$ as
well. Indeed, if this is not the case we can consider the competitor
$g:=\min\{f^\tau,C\}$ in the above minimization problem. By locality
we get ${\bf C}_q(g)\leq {\bf C}_q(f^\tau)$ and the $L^2$ distance of $f$ and $g$
is strictly smaller than the one of $f$ and $f^\tau$ as soon as
$\mm(\{f^\tau>C\})>0$, which is a contradiction. Starting from
$f_0$, iterating this procedure, and using the fact that the
implicit Euler scheme converges as $\tau\downarrow 0$ (see
\cite{Brezis73}, \cite{Ambrosio-Gigli-Savare08} for details) to the
gradient flow we get the conclusion.\\* (Energy dissipation) Since
$t\mapsto f_t\in L^2(X,\mm)$ is locally absolutely continuous and,
by the maximum principle, $f_t$ take their values in $[c,C]$
$\mm$-a.e., from the fact that $\Phi$ is Lipschitz in $[c,C]$ we get
the claimed absolute continuity statement. Now notice that we have
$\tfrac{\d}{\d t}\int \Phi(f_t) \,\d\mm=\int \Phi'(f_t)\Delta_q
f_t\,\d\mm$ for a.e. $t>0$. Since $\Phi'$ belongs to $C^1([c,C])$,
from \eqref{eq:delta1} with $g=\Phi'(f_t)$ we get the conclusion.
\end{proof}

We start with the following proposition, which relates energy
dissipation to a (sharp) combination of $q$-weak gradients and
metric dissipation in $W_p$.

\begin{proposition}\label{prop:boundweak}
Assume that $\mm$ is a finite measure, let $\mu_t=f_t\mm$ be a curve in $AC^p([0,1],(\prob X,W_p))$. Assume
that for some $0<c<C<\infty$ it holds $c\leq f_t\leq C$ $\mm$-a.e.
in $X$ for any $t\in[0,1]$, and that $f_0$ is Sobolev along $q$-a.e.
curve with $\weakgradq{f_0}q\in L^q(X,\mm)$. Then for all $\Phi\in
C^2([c,C])$ convex it holds
\[
\int \Phi(f_0)\,\d\mm-\int\Phi(f_t)\,\d\mm\leq
\frac1q\iint_0^t\bigl(\Phi''(f_0)|\nabla f_0|_{w,q}\bigr)^qf_s\,\d
s\,\d\mm+\frac1p\int_0^t|\dot\mu_s|^p\,\d s\qquad\forall t>0.
\]
\end{proposition}
\begin{proof} Let $\ppi\in\prob{C([0,1],X)}$ be a plan associated to the curve
$(\mu_t)$ as in Proposition~\ref{prop:lisini}. The assumption
$f_t\leq C$ $\mm$-a.e. and the fact that
$\iint_0^1|\dot\gamma_t|^p\,\d
t\,\d\ppi(\gamma)=\int|\dot\mu_t|^p\,\d t<\infty$ guarantee that
$\ppi$ is a $p$-test plan. Now notice that it holds
$\weakgradq{\Phi'(f_0)}q=\Phi''(f_0)\weakgradq{f_0}q$ (it follows
easily from the characterization \eqref{eq:pointwisewug}), thus we
get
\[
\begin{split}
\int \Phi(f_0)-\int\Phi(f_t)\,\d\mm&\leq
\int \Phi'(f_0)(f_0-f_t)\,\d\mm=\int \Phi'(f_0)\circ \e_0-\Phi'(f_0)\circ \e_t\,\d\ppi\\
&\leq\iint_0^t\Phi''(f_0(\gamma_s))\weakgradq{f_0}q(\gamma_s)|\dot\gamma_s|\,\d s\,\d\ppi(\gamma)\\
&\leq\frac1q\iint_0^t\bigl(\Phi''(f_0(\gamma_s))|\nabla
f_0|_{w,q}(\gamma_s)\bigr)^q\,\d s\,\d\ppi(\gamma)
+\frac1p\iint_0^t|\dot\gamma_s|^p\,\d s\,\d\ppi(\gamma)\\
&=\frac1q\iint_0^t\bigl(\Phi''(f_0)|\nabla f_0|_{w,q}\bigr)^qf_s\,\d
s\,\d\mm+\frac1p\int_0^t|\dot\mu_s|^p\,\d s.
\end{split}
\]
\end{proof}

The key argument to achieve the identification is the following
lemma which gives a sharp bound on the $W_p$-speed of the
$L^2$-gradient flow of ${\bf C}_q$. This lemma has been introduced in
\cite{Kuwada10} and then used in
\cite{GigliKuwadaOhta10,Ambrosio-Gigli-Savare11} to study the heat
flow on metric measure spaces.

\begin{lemma}[Kuwada's lemma]\label{le:kuwada}
Assume that $\mm$ is a finite measure, let $f_0\in L^2(X,\mm)$ and let $(f_t)$ be the gradient flow of
${\bf C}_q$ starting from $f_0$. Assume that for some $0<c<C<\infty$ it
holds $c\leq f_0\leq C$ $\mm$-a.e. in $X$, and that $\int
f_0\,\d\mm=1$. Then the curve $t\mapsto \mu_t:=f_t\mm\in\prob X$ is
absolutely continuous w.r.t. $W_p$ and it holds
\[
|\dot\mu_t|^p\leq\int\frac{|\nabla f_t|_{*,q}^q}{f_t^{p-1}}\,\d
\mm\qquad\text{for a.e. $t\in (0,\infty)$.}
\]
\end{lemma}
\begin{proof}
We start from the duality formula \eqref{eq:dualitabase} (written
with $\varphi=-\psi$)
\begin{equation}\label{eq:dualityQ}
\frac{W_p^p(\mu,\nu)}p=\sup_{\varphi\in{\rm Lip}_b(X)}\int_X
Q_1\varphi\, d\nu-\int_X\varphi\,d\mu.
\end{equation}
where $Q_t\varphi$ is defined in \eqref{eq:Nicola1} and
\eqref{eq:Nicola2}, so that $Q_1\varphi=\psi^c$. 

We prove that the duality formula \eqref{eq:dualityQ} is still true if the supremum in the right-hand side is taken over nonnegative 
and bounded $\varphi \in {\rm Lip}(X)$ with bounded support
\begin{equation}\label{eq:dualityQbounded}
\frac{W_p^p(\mu,\nu)}p=\sup\Big\{\int_X
Q_1\varphi\, d\nu-\int_X\varphi\,d\mu : \varphi \in {\rm Lip}(X), \, \varphi\geq 0,\text{ with bounded support} \Big\}.
\end{equation}
The duality formula \eqref{eq:dualityQ} holds also if the supremum is taken over bounded nonnegative $\varphi$ in ${\rm Lip}(X)$ up to a translation.  
In order to prove the equivalence it is enough to show that for every $\varphi \in {\rm Lip}(X)$ bounded and nonnegative
\begin{equation}\label{eqn:conv-phi}
\lim_{r \to \infty} \Big\{ \int_X
Q_1[\chi_r\varphi]\, d\nu-\int_X\chi_r\varphi\,d\mu \Big\}= \int_X
Q_1\varphi\, d\nu-\int_X\varphi\,d\mu,
\end{equation}
where $\chi_r$ is a cutoff function which is nonnegative, $1$ in $B(x_0,r)$ and $0$ outside $B(x_0,r+1)$ for some $x_0 \in X$ fixed.
By the dominated convergence theorem we have that
$$\lim_{r \to \infty}\int_X\chi_r\varphi\,d\mu=\int_X\varphi\,d\mu.$$

\null
From \eqref{eq:Nicola2} it follows that $Q_1[\chi_r \varphi]$ is nonnegative and 
$$
Q_1[\chi_r \varphi](x)=\inf_{y\in X} \left\{\chi_r(y)\varphi(y)+\frac{\sfd^p(x,y)}{p} \right\} \leq \inf_{y\in X} \left\{\varphi(y)+\frac{\sfd^p(x,y)}{p} \right\} = Q_1\varphi(x).
$$
Moreover, setting $B_{r,\varphi}$ the ball of center $x_0$ and radius $r-({\rm Lip(\varphi)})^{1/p}$, we have that  
$Q_1[\chi_r \varphi]= Q_1\varphi$ in  $B_{r,\varphi}$. Hence
\begin{equation}\label{eqn:conv-Qphi}
 \Big| \int_X Q_1[\chi_r\varphi]\, d\nu - \int_X
Q_1\varphi\, d\nu \Big| = \int_{B^c_{r,\varphi}} |Q_1[\chi_r\varphi] -
Q_1\varphi| \, d\nu \leq 2  \int_{B^c_{r,\varphi}} Q_1\varphi \, d\nu
\end{equation}
and the last term goes to $0$ as $r\to \infty$.
From \eqref{eqn:conv-phi} and \eqref{eqn:conv-Qphi} we obtain \eqref{eq:dualityQbounded}.

Fix now $\varphi\in{\rm
Lip}(X)$ nonnegative with bounded support and recall that $Q_t \phi$ has bounded support for every $t>0$ 
and that (Proposition~\ref{prop:timederivative}) the
map $t\mapsto Q_t\varphi$ is Lipschitz with values in $C(X)$, in
particular also as a $L^2(X,\mm)$-valued map.

Fix also $0\leq t<s$, set $\ell=(s-t)$ and recall that since $(f_t)$
is a gradient flow of ${\bf C}_q$ in $L^2(X,\mm)$, the map $[0,\ell]\ni
\tau\mapsto f_{t+\tau}$ is absolutely continuous with values in
$L^2(X,\mm)$. Therefore, since both factors are uniformly bounded,
the map $[0,\ell]\ni\tau\mapsto Q_{\frac\tau\ell}\varphi f_{t+\tau}$
is absolutely continuous with values in $L^2(X,\mm)$. In addition,
the equality
\[
\frac{Q_{\frac{\tau+h}\ell}\varphi
f_{t+\tau+h}-Q_{\frac{\tau}\ell}\varphi
f_{t+\tau}}{h}=f_{t+\tau}\frac{Q_{\frac{\tau+h}\ell}-Q_{\frac\tau\ell}\varphi
}{h}+Q_{\frac{\tau+h}\ell}\varphi\frac{ f_{t+\tau+h}-
f_{t+\tau}}{h},
\]
together with the uniform continuity of $(x,\tau)\mapsto
Q_{\frac\tau\ell}\varphi(x)$ shows that the derivative of
$\tau\mapsto Q_{\frac\tau\ell}\varphi f_{t+\tau}$ can be computed
via the Leibniz rule.

We have:
\begin{equation}
\label{eq:step1}
\begin{split}
\int_X Q_1\varphi\,\d\mu_s-\int_X\varphi \,\d\mu_t&
=\int Q_1\varphi f_{t+\ell}\,\d\mm-\int_X\varphi f_t\,\d\mm
=\int_X\int_0^\ell\frac{\d}{\d\tau}\big(Q_{\frac\tau\ell}\varphi f_{t+\tau}\big)d\tau \,\d\mm\\
&\leq\int_X\int_0^\ell 
-\frac{{\rm Lip}^q_a(Q_{\frac\tau\ell}\varphi,\cdot)}{q\ell}f_{t+\tau}+
Q_{\frac\tau\ell}\varphi \Delta_q f_{t+\tau}\,\d\tau \,\d\mm,\\
\end{split}
\end{equation}
having used Theorem~\ref{thm:subsol}.

Observe that by inequalities \eqref{eq:delta1} and \eqref{eq:facile}
we have
\begin{equation}
\label{eq:sarannouguali}
\begin{split}
\int_X Q_{\frac\tau\ell}\varphi \Delta_q f_{t+\tau} \,\d\mm& \leq
\int_X\relgradq{Q_{\frac\tau\ell}\varphi}q|\nabla
f_{t+\tau}|_{*,q}^{q-1}\,\d\mm\leq \int_X{\rm Lip}_a(
Q_{\frac\tau\ell}\varphi)|\nabla f_{t+\tau}|_{*,q}^{q-1} \,\d\mm\\
&\leq \frac1{q\ell}\int_X{\rm Lip}_a^q(Q_{\frac\tau\ell}\varphi,\cdot)
f_{t+\tau}d\mm+\frac{\ell^{p-1}} p\int_X\frac{|\nabla
f_{t+\tau}|_{*,q}^q}{f_{t+\tau}^{p-1}}\,\d\mm.
\end{split}
\end{equation}
Plugging this inequality in \eqref{eq:step1}, we obtain
\[
\int_X Q_1\varphi \,\d\mu_s-\int_X\varphi \,\d\mu_t\leq
\frac{\ell^{p-1}} p\int_0^\ell\int_X\frac{|\nabla
f_{t+\tau}|_{*,q}^q}{f_{t+\tau}^{p-1}}\,\d\mm.
\]
This latter bound does not depend on $\varphi$, so from
\eqref{eq:dualityQ} we deduce
\[
W_p^p(\mu_t,\mu_s)\leq \ell^{p-1}\int_0^\ell\int_X\frac{|\nabla
f_{t+\tau}|_{*,q}^q}{f^{p-1}_{t+\tau}}\,\d\mm.
\]
At Lebesgue points of $r\mapsto\int_X|\nabla
f_r|_{*,q}^q/f_r^{p-1}\,\d\mm$ where the metric speed exists we
obtain the stated pointwise bound on the metric speed.
\end{proof}

\section{Equivalence of gradients}\label{sequivalence}

In this section we assume that $(X,\sfd)$ is complete and separable, and that
$\mm$ is finite on bounded sets. We prove the equivalence of weak gradients, considering first
the simpler case of a finite measure $\mm$. 

\begin{theorem}\label{thm:graduguali}
Let $f\in L^q(X,\mm)$. Then $f$ has a $q$-relaxed slope if and only if
$f$ has a $q$-weak upper gradient and $\relgradq fq=\weakgradq fq$ $\mm$-a.e. in $X$.
\end{theorem}
\begin{proof} One implication and the inequality $\geq$ have already been established in \eqref{allinequalities1}.
We prove the converse ones first for finite measures, and then in the general case.

So, assume for the moment that $\mm(X)<\infty$.
Up to a truncation argument and addition of a constant, we can
assume that $0<c\leq f\leq C<\infty$ $\mm$-a.e. for some $0<c\leq
C<\infty$. Let $(g_t)$ be the $L^2$-gradient flow of ${\bf C}_q$ starting
from $g_0:=f$ and let us choose $\Phi\in C^2([c,C])$ in such a way
that $\Phi''(z)=z^{1-p}$ in $[c,C]$. Recall that $c\leq g_t\leq C$
$\mm$-a.e. in $X$ and that from Proposition~\ref{prop:basecal} we
have
\begin{equation}\label{eq:Amerio}
\int\Phi(g_0)\,\d\mm-\int\Phi(g_t)\,\d\mm=\int_0^t\int_X\Phi''(g_s)|\nabla
g_s|_{*,q}^q\d\mm\,\d s\qquad\forall t\in [0,\infty).
\end{equation}
In particular this gives that $\int_0^\infty\int_X\Phi''(g_s)|\nabla
g_s|_{*,q}^q\,\d\mm\,\d s$ is finite. Setting $\mu_t=g_t\mm$,
Lemma~\ref{le:kuwada} and the lower bound on $g_t$ give that
$\mu_t\in AC^p\bigl((0,\infty),(\prob X,W_p)\bigr)$, so that
Proposition~\ref{prop:boundweak} and Lemma~\ref{le:kuwada} yield
\[
\int \Phi(g_0)\,\d\mm-\int
\Phi(g_t)\,\d\mm\leq\frac1q\int_0^t\int_X\bigl(\Phi''(g_0)|\nabla
g_0|_{w,q}\bigr)^q g_s\,\d\mm\,\d s+\frac1p\int_0^t\int_X\frac{|\nabla
g_s|_{*,q}^q}{g_s^{p-1}}\,\d\mm\,\d s.
\]
Hence, comparing this last expression with \eqref{eq:Amerio}, our
choice of $\Phi$ gives
\[
\frac1q\iint_0^t\,\frac{|\nabla g_s|_{*,q}^q}{g_s^{p-1}}\d
s\,\d\mm\leq\int_0^t\int_X\frac1q \bigl(\frac{|\nabla
g_0|_{w,q}}{g_0^{p-1}}\bigr)^q g_s\,\d\mm\,\d s.
\]
Now, the bound $f\geq c>0$ ensures $\Phi''(g_0)|\nabla g_0|_{*,q}\in
 L^q(X,\mm)$. In addition, the
maximum principle together with the convergence of $g_s$ to $g_0$ in
$L^2(X,\mm)$ as $s\downarrow 0$ grants that the convergence is also
weak$^*$ in $L^\infty(X,\mm)$, therefore
\[
\limsup_{t\downarrow 0}\frac{1}{t}\iint_0^t\,\frac{|\nabla
g_s|_{*,q}^q}{g_s^{p-1}}\d s\,\d\mm\leq\int_X \frac{|\nabla
g_0|_{w,q}^q}{g_0^{q(p-1)}}g_0\d\mm=\int_X \frac{|\nabla
g_0|_{w,q}^q}{g_0^{p-1}}\,\d\mm.
\]
The lower semicontinuity property stated in
Remark~\ref{rem:basiclsc} with $\alpha=p-1$ then gives
\[
\int_X \frac{|\nabla g_0|_{*,q}^q}{g_0^{p-1}}\,\d\mm\leq \int_X
\frac{|\nabla g_0|_{w,q}^q}{g_0^{p-1}}\,\d\mm.
\]
This, together with the inequality $\weakgradq {g_0}q\leq\relgradq
{g_0}q$ $\mm$-a.e. in $X$, gives the conclusion.

Finally, we consider the general case of a measure $\mm$ finite on bounded sets. Let $X_n=\overline{B}(x_0,n)$, $n>1$, and notice that trivially it holds
\begin{equation}\label{eq:localoca}
|\nabla f|_{X_n,w,q}\leq|\nabla f|_{w,q}\qquad\text{$\mm$-a.e. in $X_n$,}
\end{equation}
because the class of test plans relative to $X_n$ is smaller. 
Hence, if we apply the equivalence result in $X_n$, we can find Lipschitz functions $f_k:X_n\to\R$ which converge to
$f$ in $L^q(X_n,\mm)$ and satisfy ${\rm Lip}_{a,X_n}(f_k,\cdot)\to |\nabla f|_{w,X_n}$ in $L^q(X_n,\mm)$.
If $\psi_n:X\to [0,1]$ is a $2$-Lipschitz function identically equal to 1 on $\overline{B}(x_0,n-1)$ and with support
contained in $B(0,n-\tfrac{1}{4})$, the functions $\psi_n f_k$ can obviously be thought as Lipschitz functions with bounded support
on $X$ and satisfy (thanks to \eqref{eq:leibn-relgradq})
$$
{\rm Lip}_a(\psi_nf_k)\leq\psi_n{\rm Lip}_{a,X_n}(f_k)+2\chi_n|f_k|,
$$
where $\chi_n$ is the characteristic function of $\overline{B}(0,n)\setminus B(0,n-1)$. 
Passing to the limit as $k\to\infty$ (notice that multiplication by $\psi_n$ allows to turn $L^q(X_n,\mm)$ convergence of
the asymptotic Lipschitz constants to $L^q(X,\mm)$ convergence, and similarly for $f_k$) 
it follows that $\psi_n f$ has $q$-relaxed slope, and that
$$
|\nabla (\psi_n f)|_{*,q}\leq |\nabla f|_{X_n,w,q}+2\chi_n|f|\qquad\text{$\mm$-a.e. in $X$.}
$$
Invoking \eqref{eq:localoca} we obtain
$$
|\nabla (\psi_n f)|_{*,q}\leq |\nabla f|_{w,q}+2\chi_n|f|\qquad\text{$\mm$-a.e. in $X$.}
$$
Eventually we let $n\to\infty$ to conclude, by a diagonal argument, that $f$ has a $q$-relaxed slope
and that $|\nabla f|_{*,q}\leq |\nabla f|_{w,q}$ $\mm$-a.e. in $X$.
\end{proof}

The proof of the previous result provides, by a similar argument, the following locality result.

\begin{proposition}\label{prop:locaopen}
If $f$ has a $q$-weak upper gradient and $A\subset X$ is open, then
\begin{equation}\label{eq:locaopen1}
|\nabla f|_{\overline{A},w,q}=|\nabla f|_{w,q}\qquad\text{$\mm$-a.e. in $A$.}
\end{equation}
\end{proposition}
\begin{proof}
We already noticed that, by definition, $|\nabla f|_{\overline{A},w,q}\leq |\nabla f|_{w,q}$ $\mm$-a.e. in $\overline{A}$.
Let $B\subset A$ be an open set with ${\rm dist}(B,X\setminus A)>0$ and let $\psi:X\to [0,1]$ be a Lipschitz cut-off function 
with support contained in $A$ and equal to $1$ on a neighbourhood of $B$. If $f_n\in\Lip(\overline{A})$ have bounded support,
converge to $f$ in $L^q(\overline{A},\mm)$ and satisfy $\Lip_a(f_n,\cdot)\to|\nabla f|_{\overline{A},*,q}$ in $L^q(\overline{A},\mm)$, 
we can consider the functions $f_n\psi$ and use \eqref{eq:leibn-relgradq} to obtain that $\psi|\nabla f|_{\overline{A},*,q}+\Lip(\psi)\chi|f|$ is a $q$-relaxed slope
of $f$ in $X$, where $\chi$ is the characteristic function of the set $\overline{\{\psi<1\}}$. Since $\chi\equiv 0$ on $B$ it follows that
$$
|\nabla f|_{*,q}\leq |\nabla f|_{\overline{A},*,q}\qquad\text{$\mm$-a.e. in $B$.}
$$ 
Letting $B\uparrow A$ and using the identification of gradients the proof is achieved.
\end{proof}

In particular, since any open set $A\subset X$ can be written as the increasing union of open subsets $A_n$ with $\overline{A}_n\subset A$,
it will make sense to speak of the weak gradient on $A$ of a function $f:A\to\R$ having a weak gradient when restricted to $\overline{A}_n$
for all $n$; suffices to define $|\nabla f|_{w,q}:A\to [0,\infty)$ by
\begin{equation}\label{eq:locaopen}
|\nabla f|_{w,q}:=|\nabla f|_{\overline{A}_n,w,q}\qquad\text{$\mm$-a.e. on $A_n$}
\end{equation}
and the definition is well posed $\mm$-a.e. in $X$ thanks to Proposition~\ref{prop:locaopen}.

\section{Reflexivity of $W^{1,q}(X,\sfd,\mm)$, $1<q<\infty$}

In this section we prove that the Sobolev spaces $W^{1,q}(X,\sfd,\mm)$ are reflexive when $1<q<\infty$, $(X,\sfd)$ is doubling and separable,
and $\mm$ is finite on bounded sets.
Our strategy is to build, by a finite difference scheme, a family of functionals which provide a discrete approximation of Cheeger's energy.
The definition of the approximate functionals relies on the existence of nice partitions of doubling metric spaces. 

\begin{lemma}\label{lemma:dec}
For every $\delta >0$ there exist $\ell_\delta\in\N\cup\{\infty\}$ and pairs 
set-point $(A_i^{\delta},  z^{\delta}_i)$, $0\leq i<\ell_\delta$, where $A_i^{\delta} \subset X$ are Borel sets and $z_i^{\delta}\in X$, 
satisfying:
\begin{itemize}
 \item[(i)] the sets $A_i^\delta$, $0\leq i<\ell_\delta$, are a partition of $X$ and $\sfd(z_i^\delta,z_j^\delta)>\delta$ whenever $i\neq j$;
 \item[(ii)] $A_i^{\delta}$ are comparable to balls centered at $z_i^\delta$, namely
 $$B \left(z_i^{\delta},\frac{\delta}{3}\right) \subset A_i^{\delta} \subset B \left(z_i^{\delta},\frac 54 \delta\right).$$
\end{itemize}
\end{lemma}
\begin{proof} Let us fix once for all a countable dense set $\{x_k\}_{k \in \N}$. Then, starting from $z_0^\delta=x_0$, we proceed in this way:
\begin{itemize}
 \item for $i\geq 1$, set recursively
 $$ B_i = X \setminus \bigcup_{j <i} \overline{B} (z_j^{\delta}, \delta );$$
 \item if $B_i=\emptyset$ for some $i\geq 1$, then the procedure stops. Otherwise, take $z_i^{\delta} = x_{k_i}$ where
 $$ k_i= \min \{ k \in \N \; : \;  x_k \in B_i \}. $$
\end{itemize}
We claim that for every $\ep >0$ we have that
$$ \bigcup_{i= 0}^{\infty} B( z_i^{\delta} , \delta + \varepsilon) = X.$$
To show this it is sufficient to note that for every $x \in X$ we have a point $x_j$ such that $\sfd (x_j,x) < \ep$; then either $x_j=z_i^{\delta}$ 
for some $i$ or $x_j \in \overline{B}(z_i^{\delta}, \delta)$. In both cases we get
\begin{equation}\label{eqn:exist} 
\forall x \in X \; \exists i \in \N \qquad \text{such that $\sfd( z_i^{\delta} , x) < \delta + \ep$.}
\end{equation}
Now we define the sets $A_i^{\delta}$ similarly to a Voronoi diagram constructed from the starting point $z_i^{\delta}$: 
$$ A_0^{\delta} = \left\{ x \in X \; : \; \sfd ( x, z_0^{\delta} ) \leq \sfd ( x, z_j^{\delta} )+ \ep \quad \forall j > 0 \right\}, $$
$$ A_i^{\delta} = \Big\{ x \in X \setminus \Big( \bigcup_{j < i } A_j^{\delta} \Big) \; : \; \sfd ( x, z_i^{\delta} ) \leq \sfd ( x, z_j^{\delta} )+ \ep \quad \forall j > i \Big\}. $$ 
By construction all these sets are Borel and disjoint. We can also give a dual definition: $x \in A_{k}^{\delta}$ iff
$$ k= \min I_x\qquad\text{where}\qquad I_x=\left\{ i \in \N \; : \; \sfd ( x, z_i^{\delta} ) \leq \sfd ( x, z_j^{\delta} )+ \ep \quad \forall j \in \N \right\}. $$
In other words, we're minimizing the quantity $ \sfd ( x, z_i^{\delta})$ and among those indeces $i$ who are minimizing up to $\ep$ we take the least one $i_x$. This proves that $I_x$ is non empty and by this quasi minimality and \eqref{eqn:exist} we obtain $ \sfd ( x, z_{i_x}^{\delta} ) \leq \inf_{i\in \N} \sfd ( x, z_{i}^{\delta}) + \ep < \delta + 2 \ep$. Furthermore if $ \sfd ( x, z_i^{\delta} ) < \delta /2 - \ep /2 $ then $I_x = \{ i \}$. Indeed, suppose there is another $j \in I_x$ with $j \neq i$, then $\sfd ( z_j^{\delta} ,x ) \leq \sfd (z_i^{\delta},x ) + \ep \leq \delta+ \ep/2$ and so
$$ \delta < \sfd ( z_i^{\delta} , z_j^{\delta} ) \leq \sfd (z_i^{\delta}, x ) + \sfd (z_j^{\delta},x ) \leq \delta. $$
We just showed that
$$ B\left( z_i^{\delta} , \frac{\delta}2 - \frac{\ep}2 \right) \subset A_i^{\delta} \subset B(z_i^{\delta} , \delta + 2 \ep ). $$
The dual definition gives us that $A_i^\delta$ are a partition of $X$, and (ii) is satisfied choosing $\ep = \delta/8$.
\end{proof}

Note that this construction is quite simpler if $X$ is locally compact, which is always the case if
$(X,\sfd)$ is doubling and complete. In this case we can choose $\ep=0$.

We remark that partitions with additional properties have also been studied in the literature. For example, in \cite{Chr} \emph{dyadic} partitions of a 
doubling metric measure space are constructed.

\begin{definition} [Dyadic partition] A dyadic partition is made by a sequence $(\ell_h)\subset\N\cup\{\infty\}$ 
and by collections of disjoint sets (called cubes) $ \Delta^h= 
 \{ A^h_i \}_{1\leq i<\ell(h)\}}$ such that for every $h\in\N$ the following properties hold:
\begin{itemize}
 \item $ \mm\bigl(X\setminus\bigcup_i A^h_i\bigr)=0$;
 \item for every $i\in [1,\ell_{h+1})$ there exists a unique $j\in [1,\ell_h)$ such that $A^{h+1}_i \subset A^h_j$;
 \item for every $i\in [1,\ell_h)$ there exists $z^h_i\in X$ such that $B(z^h_i,a_0\delta^h) \subset A^h_i\subset B(z^h_i, a_1 \delta^h)$ for some 
positive constants $\delta,\,a_0,\,a_1$ independent of $i$ and $h$.
\end{itemize}
\end{definition}

In \cite{Chr} existence of dyadic decompositions is proved, with $\delta$, $a_1$ and $a_0$ depending on the constant $\tilde{c}_D$
in \eqref{eqn:doubll}.
Although some more properties of the partition might give additional information on the functionals that we are going to construct, 
for the sake of simplicity we just work with the partition given by Lemma~\ref{lemma:dec}.

In order to define our discrete gradients we give more terminology.
We say that $A^{\delta}_i$ is a \emph{neighbor} of $A^{\delta}_j$, and we denote by $A^{\delta}_i  \sim A^{\delta}_j$,  
if their distance is less than $\delta$. In particular $A^{\delta}_i  \sim A^{\delta}_j$ implies that $\sfd(z_i^{\delta},z_j^{\delta}) \leq 4 \delta$:
indeed, if $\tilde{z}^\delta_i\in A^\delta_i$ and $\tilde{z}^\delta_j\in A^\delta_j$ satisfy $\sfd(\tilde{z}_i^\delta,\tilde{z}_j^\delta)<\delta'$ we have
$$
\sfd(z^\delta_i,z^\delta_j)\leq\sfd(z^\delta_i,\tilde{z}_i^\delta)+\sfd(\tilde{z}_i^\delta,\tilde{z}_j^\delta)+\sfd(\tilde{z}_j^\delta,z_j^\delta)\leq
\frac{10}{4}\delta+\delta'
$$
and letting $\delta'\downarrow\delta$ we get $$\sfd(z_i^{\delta},z_j^{\delta}) \leq \frac{14}{4}\delta\leq 4 \delta.$$
This leads us to the first important property of doubling spaces:
\begin{equation}\label{rmk:finite-neighd}
\text{In a $c_D$-doubling metric space $(X,\sfd)$, every $A^\delta_i$ has at most $c_D^3$ neighbors.}
\end{equation}
Indeed, we can cover $B(z_i^{\delta},4\delta)$ with $c_D^3$ balls with radius $\delta/2$ but each 
of them, by the condition $\sfd(z_i^\delta,z_j^\delta)>\delta$, can contain only one of the $z_j^{\delta}$'s.

Now we fix $\delta \in (0,1)$ and we consider a partition $A_i^{\delta}$ of $\supp\mm$ on scale $\delta$.
For every  $u \in L^q(X, \mm)$ we define the average $u_{\delta,i}$ of $u$ in each cell of the partition by $ \fint_{A^\delta_i} u\,\d\mm$. 
We denote by $\PCh{\delta}(X)$, which depends on the chosen decomposition as well, the set of functions $u \in L^q(X, \mm)$ constant on each 
cell of the partition at scale $\delta$, namely
$$u(x) = u_{\delta,i} \qquad \mbox{for $\mm$-a.e.\ } x\in {A^\delta_i}.$$
We define a linear projection functional $\Ph{\delta}:  L^q(X, \mm) \to  \PCh{\delta}(X)$ by
 $\Ph{\delta} u (x) = u_{\delta,i}$ for every $x \in A^\delta_i$.

The proof of the following lemma is elementary.

\begin{lemma} \label{lem:contractivity}
$\Ph{\delta}$ are contractions in $L^q(X,\mm)$ and $\Ph{\delta}u\to u$ in $L^q(X,\mm)$ as $\delta\downarrow 0$ for
all $u\in L^q(X,\mm)$. 
\end{lemma}

Indeed, the contractivity of $\Ph{\delta}$ is a simple consequence of Jensen's inequality and it suffices to check the convergence
of $\Ph{\delta}$ as $\delta\downarrow 0$ on a dense subset of $L^q(X,\mm)$. Since $\mm$ is finite on bounded sets, suffices to
consider bounded continuous functions with bounded support. Since bounded closed sets are compact,
by the doubling property, it follows that any such function $u$ is uniformly continuous, so that $\Ph{\delta}u\to u$ pointwise
as $\delta\downarrow 0$. Then, we can use the dominated convergence theorem to conclude.

We now define an  approximate gradient as follows: it
 is constant on the cell $A_i^\delta$ for every $\delta,\,i\in \N$ and it takes the value
$$ |\Dh{\delta} u |^q (x) :=\frac 1{\delta^q} \sum_{A^\delta_j \sim A^\delta_i} | u_{\delta,i} - u_{\delta,j}|^q \qquad \forall x \in A^\delta_i.$$
We can accordingly define the functional $\Fh{\delta,q}: L^q(X, \mm)\to [0,\infty]$ by
 \begin{equation}\label{defn:fh}
  \Fh{\delta,q} (u) := \int_X | \Dh{\delta} u | ^q (x) \,\d \mm (x).
 \end{equation}
 
 Now, using the weak gradients, we define a functional $\Ch:L^q(X,\mm)\to [0,\infty]$ that
 we call \emph{Cheeger} energy, formally similar to the
 one \eqref{def:Cheeger} used in Section~\ref{sec:Chq}, for the purposes of energy dissipation estimates and 
 equivalence of weak gradients. Namely, we set
 $$
 \Ch_q(u):=
 \begin{cases}
 \int_X|\nabla u|_{w,q}^q\,\d\mm &\text{if $u$ has a $q$-relaxed slope}\cr
 +\infty &\text{otherwise.}
 \end{cases}
 $$

At this level of generality, we cannot expect that the functionals $\Fh{\delta,q}$
$\Gamma$-converge as $\delta\downarrow 0$. However, since $L^q(X, \mm)$ is a complete 
and separable metric space, from the compactness property of $\Gamma$-convergence stated in Proposition~\ref{prop:Gammacompact}
we obtain that the 
functionals $\Fh{\delta,q}$ have $\Gamma$-limit points as $\delta\downarrow 0$.
 
\begin{theorem}\label{thm:main}
Let $(X,\sfd,\mm)$ be a metric measure space with $(\supp\mm,\sfd)$ complete and doubling, $\mm$ finite on bounded sets.
Let $\F$ be a $\Gamma$-limit point of $\Fh{\delta,q}$ as $\delta\downarrow 0$, namely
$$ \F:= \Gamma \!- \! \lim_{k\to \infty} \mathcal{F}_{\delta_k,q}, $$
for some infinitesimal sequence $(\delta_k)$, where the $\Gamma$-limit is computed with respect to the $L^q(X,\mm)$ distance.
Then:
\begin{enumerate}
\item[(a)] $\F$ is equivalent to the Cheeger energy $\Ch_q$, namely there exists $\eta=\eta(q,c_D)$ such that
\begin{equation}\label{ts:est-f-ch}
\frac{1}{\eta} \, \Ch_q (u) \leq \F (u) \leq \eta \, \Ch_q (u)\qquad\forall u\in L^q(X,\mm).
\end{equation}
\item[(b)] The norm on $W^{1,q}(X,\sfd,\mm)$ defined by
\begin{equation}\label{eqn:equiv-metr}
\left( \|u\|^q_q+\F(u) \right)^{1/q} \qquad \forall u\in W^{1,q}(X,\sfd,\mm)
\end{equation}
is uniformly convex.
Moreover, the seminorm $\Ftwo^{1/2}$ is Hilbertian, namely 
\begin{equation}\label{eqn:quadr}
\Ftwo(u+v) + \Ftwo(u-v) = 2 \big(\Ftwo(u)+\Ftwo(v) \big) \qquad \forall u,\,v \in W^{1,2}(X,\sfd,\mm).
\end{equation}
\end{enumerate}
\end{theorem}

\begin{corollary}[Reflexivity of $W^{1,q}(X,\sfd,\mm)$]\label{cor:refl}
Let $(X,\sfd,\mm)$ be a metric measure space with $(\supp\mm,\sfd)$ doubling and $\mm$ finite on bounded sets.
The Sobolev space $W^{1,q}(X,\sfd,\mm)$ of functions $u\in L^q(X,\mm)$ with a $q$-relaxed slope, 
endowed with the usual norm
\begin{equation}\label{eqn:Sob-metr}
\left( \|u\|_q^q+\Ch_q(u) \right)^{1/q} \qquad \forall u\in W^{1,q}(X,\sfd,\mm),
\end{equation}
is reflexive. \end{corollary}
 \begin{proof}  Since the Banach norms \eqref{eqn:equiv-metr} and  \eqref{eqn:Sob-metr} on $W^{1,q}(X,\sfd,\mm)$ are equivalent thanks to \eqref{ts:est-f-ch}
 and reflexivity is invariant, we can work with the first norm. The Banach space $W^{1,q}(X,\sfd,\mm)$ endowed with the first norm is reflexive 
 by uniform convexity and Milman-Pettis theorem. 
 \end{proof}
 
We can also prove, by standard functional-analytic arguments, that reflexivity implies separability.

\begin{proposition}[Separability of $W^{1,q}(X,\sfd,\mm)$]
 If $W^{1,q}(X,\sfd,\mm)$ is reflexive, then it is separable and bounded Lipschitz functions with bounded support are dense.
 \end{proposition}
 \begin{proof}
 The density of Lipschitz functions with bounded support follows at once from the density
 of this convex set in the weak topology, ensured by Proposition~\ref{prop:easy}. In order to prove separability, suffices to consider 
 for any $M$ a countable and $L^q(X,\mm)$-dense subset ${\cal D}_M$ of 
 $${\cal L}_M:=\left\{f\in\Lip(X)\cap L^q(X,\mm):\ \int_X|\nabla f|^q\,\d\mm\leq M\right\},$$ 
 stable under convex combinations with rational coefficients. The weak closure of ${\cal D}_M$ obviously contains ${\cal L}_M$,
 by reflexivity (because if $f_n\in{\cal D}_M$ converge to $f\in{\cal L}_M$ in $L^q(X,\mm)$, then $f_n\to f$ weakly in $W^{1,q}(X,\sfd,\mm)$); 
 being this closure convex, it coincides with the strong closure of ${\cal D}_M$.
 This way we obtain that the closure in the strong topology of $\cup_M{\cal D}_M$ contains all Lipschitz functions with bounded support. 
\end{proof}

The strategy of the proof of statement (a) in Theorem~\ref{thm:main} consists in proving the estimate from above of $\F$ with
relaxed gradients and the estimate from below with weak gradients. Then, the equivalence between weak and relaxed gradients
provides the result. In the estimate from below it will be useful the discrete version of the $q$-weak upper gradient property given in
Definition~\ref{dfn:wug-discr}.

In the following lemma we prove that for every $u \in L^q(X,\mm)$ we have that $4|\Dh{\delta} u | $ is a $q$-weak 
upper gradient for $\Ph{\delta} u$ up to scale $\delta/2$.

\begin{lemma} \label{lem:ugdiscrete} Let $\gamma \in AC^p([0,1];X)$. Then we have that
\begin{equation}\label{eqn:grad-discr-a-b}
|\Ph{\delta} u ( \gamma_b) - \Ph{\delta} u ( \gamma_a)  | \leq 4
\int_a^b  | \Dh{\delta} u| ( \gamma_t)  | \dot {\gamma}_t | \, \d t 
\qquad\text{for all $a<b$  s.t. $\int_a^b |\dot{\gamma}_t|\,\d t > {\delta}/2$.}
\end{equation}
In particular $4|\Dh{\delta} u|$ is a $q$-weak upper gradient of $\Ph{\delta}u$ up to scale $\delta/2$.
\end{lemma}
\begin{proof}
It is enough to prove the inequality under the more restrictive assumption that
\begin{equation}\label{eqn:cond-a-b}
\frac{ \delta}2  \leq \int_a^b |\dot{\gamma}_t| \,\d t  \leq  \delta,
\end{equation}
because then we can slice every interval $(a,b)$ that is longer than $\delta/2$ into subintervals that satisfy \eqref{eqn:cond-a-b}, and 
we get \eqref{eqn:grad-discr} by adding the inequalities for subintervals and using triangular inequality.

Now we prove \eqref{eqn:grad-discr} for every $a,\,b \in [0,1]$ such that \eqref{eqn:cond-a-b} holds. Take any time $t\in [a,b]$; by assumption, it is 
clear that $\sfd(\gamma_t,\gamma_a)\leq\delta$ and $\sfd(\gamma_t,\gamma_b)\leq\delta$, so that the cells relative to 
$\gamma_a$ and $\gamma_b$ are both neighbors of the one relative to $\gamma_t$. By definition then we have:
$$
|\Dh{\delta} u|^q ( \gamma_t) \geq \frac 1{ \delta^q}\left( | \Ph{\delta} u(\gamma_b) - \Ph{\delta} u(\gamma_t)|^q  
+ | \Ph{\delta} u(\gamma_t) - \Ph{\delta} u(\gamma_a)|^q  \right) 
\geq \frac 1{2^{q-1}\delta^q} | \Ph{\delta} u ( \gamma_b) - \Ph{\delta} u (\gamma_b) |^q.
$$
Taking the $q$-th root and integrating in $t$ we get
$$\int_a^b | \Dh{\delta} u| (\gamma_t)| \dot{\gamma}_t| \, \d t  \geq 
\frac {|\Ph{\delta} u ( \gamma_b) - \Ph{\delta} u ( \gamma_a)|}{ 2^{1-1/q} \delta } 
\int_a^b | \dot{ \gamma}_t | \, \d t \geq \frac{1}{2} |\Ph{\delta} u ( \gamma_b) - \Ph{\delta} u ( \gamma_a)|,
$$
which proves \eqref{eqn:grad-discr-a-b}.
\end{proof}

We can now prove Theorem~\ref{thm:main}.

\noindent {\it Proof of the first inequality in \eqref{ts:est-f-ch}.} We prove that there exists a constant $\eta_1=\eta_1(c_D)$ such that 
\begin{equation}\label{ts:easy-ineq}
\Fh{q}(u) \leq \eta_1\int_X|\nabla f|_{*,q}^q\,\d\mm \qquad \forall u \in L^q(X, \mm).
\end{equation}
Let $u:X\to \R$ be a Lipschitz function with bounded support. We prove that
\begin{equation}\label{eqn:grad-lip}
| \Dh{\delta} u|^q(x) \leq 6^qc_D^3(\Lip (u,B(x,6\delta)))^q.
\end{equation}
Indeed, let us consider $i,\, j\in [1,\ell_\delta)$ such that $A_i^\delta$ and $A_j^\delta$ are neighbors.
For every $x\in A^\delta_i$, $y\in A^\delta_j$ we have that $\sfd(x,y) \leq (10/4+10/4+1)\delta= 6\delta$ and that $y\in 
B(z_i^\delta,19\delta/4)\subset B(z^\delta_i,5\delta)$.
Hence 
$$\frac { | u_{\delta,i} - u_{\delta,j}|}{\delta} \leq \frac1{\delta\mm(A_i^\delta)\mm(A_j^\delta)}
\int_{A_i^\delta\times A_j^\delta} |u(x)-u(y)| \, d\mm(x) \, d\mm(y) 
\leq 6\Lip (u,B(z_i^\delta,{5\delta})). $$
Thanks to the fact that the number of neighbors of $A_i^h$ does not exceed $c_D^3$ (see \eqref{rmk:finite-neighd}) we obtain
$$ |\Dh{\delta} u|^q(x) \leq  6^qc_D^3 (\Lip (u,B(x,6\delta)))^q\qquad\forall x\in\supp\mm,$$
which proves \eqref{eqn:grad-lip}.

Integrating on $X$ we obtain that
$$\Fh{q} (u) \leq 6^qc_D^3 \int_X (\Lip (u, B_{6\delta} (x)))^q \,\d \mm.$$
Choosing $\delta=\delta_k$, letting $k\to\infty$ and applying the dominated convergence theorem on the right-hand side
as well as the definition of asymptotic Lipschitz constant 
\eqref{defn:asymlip} we get
$$ \F{q}( u) \leq \liminf_{k\to\infty}\mathcal{F}_{q,\delta_k} (u) \leq 6^qc_D^3\int_X \Lipa^q (u, x) \,\d\mm.$$
By approximation, Proposition~\ref{prop:easy} yields \eqref{ts:easy-ineq} with $\eta_1= 6^q c_D^3$.

\noindent {\it Proof of the second inequality in \eqref{ts:est-f-ch}.}
We consider a sequence $(u_k)$ which converges to $u$ in $L^q(X, \mm)$ with $\liminf_k \mathcal{F}_{\delta_k,q}(u_k)$ finite. 
We prove that $u$ has a $q$-weak upper gradient and that  
\begin{equation}\label{eqn:lowbou} \frac 1{4^q} \int_X|\nabla u|_{w,q}^2\,\d\mm \leq \liminf_k \mathcal{F}_{\delta_k,q} (u_k).\end{equation}
Then, \eqref{ts:est-f-ch} will follow easily from \eqref{ts:easy-ineq} and the coincidence of weak and relaxed gradients.

Without loss of generality we assume that the right-hand side is finite and, up to a subsequence not relabeled, we assume that the
$\liminf$ is a limit. Hence, the sequence $f_k:=|\Dh{\delta_k} u_k|$ is bounded in $L^q(X,\mm)$ and, by weak compactness, there exist 
$g\in L^q(X,\mm)$ and a subsequence $k(h)$ such that $f_{k(h)} \rightharpoonup g$ weakly in $L^q(X,\mm)$. 
By the lower semicontinuity of the $q$-norm with respect to the weak convergence, we have that
\begin{equation}\label{eqn:ggamma} \int_X g^q\,\d\mm \leq \liminf_{h\to\infty}\int_X f_{k(h)}^q\,\d\mm=\lim_{k\to \infty} \Fh{\delta_k,q} (u_k).
\end{equation}
We can now apply Theorem~\ref{thm:stabweak} to the functions $\bar{u}_h=\Ph{\delta_{k(h)}}(u_{k(h)})$,
which converge to $u$ in $L^q(X,\mm)$ thanks to Lemma~\ref{lem:contractivity}, and to the functions $g_h=4 f_{k(h)}$ which are
$q$-weak upper gradients of $\bar{u}_h$ up to scale $\delta_{k(h)}/2$, thanks to Lemma~\ref{lem:ugdiscrete}. 
We obtain that $4g$ is a weak upper gradient of $u$, hence $g \geq  |\nabla u |_{w,q}/4$ $\mm$ a.e. in $X$. Therefore 
\eqref{eqn:ggamma} gives
$$ \frac{1}{4^q} \int_X|\nabla u|_{w,q}^q \,\d\mm\leq   \int_X g^q\,\d\mm \leq  \lim_{k\to\infty} \Fh{\delta_k,q}(u_k). $$ 

\noindent {\it Proof of statement (b).} Let $\Nph{\delta}: L^q(X,\mm) \to [0,\infty]$ be the positively $1$-homogeneous function
$$
\Nph{\delta}(u) = \left( \|\Ph{\delta} u\|^q_q+\Fh{\delta} (u) \right)^{1/q} \qquad \forall u\in L^q(X, \mm).
$$
For $q\geq 2$ we prove that $\Nph{\delta}$ satisfies the first Clarkson inequality \cite{Hewitt-Stronberg}
\begin{equation}\label{eqn:clark}
\Nph{\delta}^q\left(\frac{u+v}{2}\right) + \Nph{\delta}^q\left(\frac{u-v}{2}\right) \leq \frac{1}{2} 
\big(\Nph{\delta}^q(u)+\Nph{\delta}^q(v) \big)
\qquad\forall u,\,v \in L^q(X, \mm).
\end{equation}
Indeed, let $X_\delta\subset\N\cup (\N\times\N)$ be the (possibly infinite) set
$$X_\delta=[1,\ell_\delta) \cup  \left\{(i,j) \in [1,\ell_\delta)\times [1,\ell_\delta):\ A^\delta_i \sim A^\delta_j\right\}$$
and let $\mm_\delta$ be the counting measure on $X_\delta$.
We consider the function $\Phi_{q,\delta}:  L^q(X, \mm) \to  L^q(X_\delta, \mm_\delta)$ defined by
\[\begin{cases}
\Phi_{q,\delta}[u](i)= (\mm(A^\delta_i))^{1/q} u_{\delta,i} &\qquad \forall i\in [1,\ell_\delta)\\
\displaystyle{\Phi_{q,\delta}[u]((i,j))= (\mm(A^\delta_i))^{1/q} \frac{u_{\delta,i} - u_{\delta,j}} {\delta} } 
&\qquad \forall (i,j)\in [1,\ell_\delta)\times [1,\ell_\delta)\,\,\text{ s.t.\ } 
A^\delta_i \sim A^\delta_j.
\end{cases}
\]
It can be easily seen that $\Phi_{q,\delta}$ is linear and that 
\begin{equation}\label{eqn:phi-norm}
\|\Phi_{q,\delta}(u)\|_{L^q(X_\delta, \mm_\delta)} = \Nph{\delta}(u) \qquad \forall u\in L^q(X, \mm).
\end{equation}

Writing the first Clarkson inequality in the space $L^q(X_h, \mm_h)$ and using the linearity of $\Phi_{q,\delta}$ we immediately
obtain \eqref{eqn:clark}.
%
%
Let $\omega:(0,1)\to(0,\infty)$ be the increasing and continuous modulus of continuity $\omega(r) = 1-(1-r^q/2^q)^{1/q}$. 
From \eqref{eqn:clark} it follows that for all $u,\,v\in L^q(X, \mm)$ with $\Nph{\delta}(u)=\Nph{\delta}(v)=1$ it holds 
$$
\Nph{\delta}\left(\frac{u+v}{2}\right) \leq 1- \omega(\Nph{\delta}\left({u-v}\right)).$$
Hence $\Nph{\delta}$ are uniformly convex with the same modulus of continuity $\omega$. Thanks to Lemma~\ref{lemma:unif-conv} 
we conclude that also the $\Gamma$-limit of these norms, namely \eqref{eqn:equiv-metr}, is uniformly convex with the same modulus of continuity.

If $q<2$ the proof can be repeated substituting the first Clarkson inequality \eqref{eqn:clark} with the second one 
$$\left[\Nph{\delta}\left(\frac{u+v}{2}\right)\right]^p + 
\left[\Nph{\delta}\left(\frac{u-v}{2}\right)\right]^p \leq \left[\frac{1}{2} \bigl(\Nph{\delta}(u))^q+\frac{1}{2} \bigl(\Nph{\delta}(v)\bigr)^q \right]^{1/(q-1)} 
\quad \forall u,\,v \in L^q(X, \mm)$$
where $p = q/(q-1)$, see \cite{Hewitt-Stronberg}. In this case the modulus $\omega$ is $1-(1-(r/2)^p)^{1/p}$.

Finally, let us consider the case $q=2$. From the Clarkson inequality we get
\begin{equation}\label{eqn:quadr-1}
\Ftwo\left(\frac{u+v}{2}\right) + \Ftwo\left(\frac{u-v}{2}\right)\leq 2 \big(\Ftwo(u)+\Ftwo(v) \big).
\end{equation}
If we apply the same inequality to $u=(u'+v')/2$ and $v=(u'-v')/2$ we obtain a converse inequality and, since
$u'$ and $v'$ are arbitrary, the equality.

\section{Lower semicontinuity of the slope of Lipschitz functions}

Let us recall, first, the formulation of the Poincar\'e inequality in metric measure spaces.

\begin{definition}\label{dfn:PI}
The metric measure space $(X,\sfd ,\mm)$ supports a weak $(1,q)$-Poincar\'e inequality if there exist 
constants $ \tau ,\, \Lambda >0$ such that for every $u \in W^{1,q}(X,\sfd, \mm)$ and for every $x\in\supp\mm$, $r>0$ the following holds:
\begin{equation}\label{eqn:PI}
\fint_{ B(x,r)} \biggl| u - \fint_{B(x,r)}u \biggr|\,\d\mm \leq \tau \, r \,\left( \fint_{B_{\Lambda r}(x)} | \nabla u |^q_{w,q} \,\d\mm \right)^{1/q} .
\end{equation} 
\end{definition}

Many different and equivalent formulations of \eqref{eqn:PI} are possible: for instance we may replace in the right hand
side $|\nabla u|_{w,q}^q$ with $|\nabla u|^q$, requiring the validity of the inequality for Lipschitz functions only. The equivalence
of the two formulations has been first proved in \cite{Heinonen-Koskela99}, but one can also use the equivalence of weak and relaxed gradients
to establish it. Other formulations involve the median, or replace the left hand side by
$$
\inf_{m\in\R}\fint_{ B(x,r)} \bigl| u - m \bigr|\,\d\mm\,.
$$
The following lemma contains the fundamental estimate to prove our result.

\begin{lemma}\label{lem:telescopic} Let $(X,\sfd,\mm)$ be a doubling metric measure space 
which supports a weak $(1,q)$-Poincar\'e inequality with constants $\tau,\,\Lambda$. 
Let $u \in W^{1,q}(X,\sfd,\mm)$ and let $g=|\nabla u|_{w,q}$. There exists a constant $C>0$ depending only on the doubling constant
$\tilde c_D$ and $\tau$ such that 
\begin{equation}\label{eqn:tele2} | u(x) - u(y) | \leq C\sfd(x,y)( M^{2\Lambda\sfd(x,y)}_q g (x) + M^{2\Lambda\sfd(x,y)}_q g (y) ),
\end{equation}
for every Lebesgue points $x,\,y\in X$ of (a representative of) $u$.
\end{lemma}
\begin{proof} The main estimate in the proof is the following. Denoting by $u_{z,r}$ the mean value of $u$ on $B(z,r)$,
for every $s>0$, $x,\,y \in X$ such that $B(x,s) \subset B(y,2s)$ we have that
\begin{equation}\label{eqn:fund-est}
 |u_{x,s} - u_{y,2s}| \leq C_0(\tilde{c}_D,\tau) s M^{2 \Lambda s}_q g(y).
\end{equation}
Since $\mm$ is doubling and the space supports $(1,q)$-Poincar\'e inequality, from  \eqref{eqn:doubl} we have that 
$$
| u_{x,s} - u_{y,2s} | \leq \fint_{B(x,s)} |u - u_{y,2s} | \,\d\mm  
\leq \beta  2^{ \alpha} \fint_{ B (y,2s) } |u - u_{y,2s} | \,\d\mm\\
\leq  2^{1+ \alpha}\beta \tau s \left( \fint_{ B (y, 2 \Lambda s ) } g^q \,\d\mm \right)^{1/q}
$$
and we obtain \eqref{eqn:fund-est} with $C_0= 2^{1+ \alpha}\beta \tau$.

For every $r>0$ let $s_n = 2^{-n}r$ for every $n\geq 1$. 
If $x$ is a Lebesgue point for $u$ then $u_{x,s_n} \to u(x)$ as $n\to \infty$. Hence, applying 
\eqref{eqn:fund-est} to $x=y$ and $s_n = 2^{-n}r$,  summing on $n\geq 1$ and remarking that $M^{2\Lambda s_n}_q g \leq M^{\Lambda r}_q g$, 
we get
\begin{equation}\label{eqn:u-media}
|u_{x,r} - u(x)| \leq \sum_{n=0}^{\infty} |u_{x,s_n} - u_{x,2s_n}| \leq \sum_{n=0}^{\infty} C_0 s_n M^{\Lambda r}_q g (x) = C_0 r M^{\Lambda r}_q g (x).
\end{equation}

For every $r>0$, $x,\,y$ Lebesgue points of $u$ such that $B(x,r) \subset B(y,2r)$, we can use the triangle inequality, 
\eqref{eqn:fund-est} and \eqref{eqn:u-media} to get
\begin{align*}
 |u (x)  - u(y)| & \leq |u(x)- u_{x,r} | + |u_{x,r}- u_{y,2r} | + |u_{y,2r}-u(y)| \\
& \leq C_0 r M^{\Lambda r}_q g(x) +C_0 r M^{2 \Lambda r}_q g(y) + C_0 r M^{\Lambda r}_q g(y).
\end{align*}
Taking $r= \sfd (x,y)$ (which obviously implies $B(x,r) \subset B(y,2r)$) and since $M^{\ep}_q f (x)$ is nondecreasing in $\ep$ 
we obtain \eqref{eqn:tele2} with $C=2C_0$. 
\end{proof}

\begin{proposition}\label{prop:slope-weak} Let $(X,\sfd,\mm)$ be a doubling metric measure space, supporting a 
weak $(1,q)$-Poincar\'e inequality with constants $\tau,\,\Lambda$ and with $\supp\mm=X$
There exists a constant $C>0$ depending only on the
doubling constant $\tilde c_D$ and $\tau$ such that  
\begin{equation}\label{eqn:slope-weakgr-comp}
| \nabla u |\leq C \, |\nabla u |_{w,q} \quad\text{$\mm$-a.e. in $X$}
\end{equation}
for any Lipschitz function $u$ with bounded support.
\end{proposition}
\begin{proof} We set $g= |\nabla u|_{w,q}$; we note that $g$ is bounded and with bounded support, thus $M^{\ep}_q g$ converges to $g$ in $L^q(X,\mm)$. Let 
us fix $\lambda>0$ and a Lebesgue point  $x$ for $u$ where \eqref{eqn:leb1} is satisfied by $M^\lambda_q g$. Let $y_n \to x$ be such that
\begin{equation}\label{eqn:slope}| \nabla u | (x) =\lim_{n \to \infty} \frac{ |u(y_n)-u(x)| } {\sfd(y_n,x)}
\end{equation}
and set $r_n=\sfd(x,y_n)$, $B_n=B(y_n,\lambda r_n)\subset B(x,2r_n)$.  
Since \eqref{eqn:tele2} of Lemma~\ref{lem:telescopic} holds for $\mm$-a.e. $y \in B_n$, 
from the monotonicity of $M^{\ep}_q g $ we get
\begin{equation*}\begin{split} | u(x) - u(y_n) | &\leq \fint_{B_n} |u(x) - u(y) | \,\d \mm + \lambda r_n\Lip(u, B_n) \\
& \leq Cr_n\left( M^{4\Lambda r_n}_q g (x) + \fint_{B_n} M^{4\Lambda r_n}_q g (y) \,\d\mm(y)\right) + \lambda r_n L,
\end{split}
\end{equation*}
where $L$ is the Lipschitz constant of $u$.
For $n$ large enough $B_n \subset B(x,1)$ and $4\Lambda r_n \leq\lambda$. Using monotonicity once more we get
\begin{equation}\label{eqn:ali2}
\begin{split} | u(x) - u(y_n) | & \leq C r_n\left( M^{\lambda}_q g (x) + \fint_{B_n} M^{\lambda}_q g \,\d\mm \right) + 
\lambda r_n L
\end{split}
\end{equation}
for $n$ large enough.
Since $B(y_n,r_n)=B_n\subset B(x,2r_n)$ and since $x$ is a $1$-Lebesgue point for $M^\lambda_q g$, 
we apply \eqref{eqn:leb2} of Lemma~\ref{lem:leb} to the sets $B_n$ to get
\begin{equation}\label{eqn:leb-points}
 \lim_{n\to\infty}\fint_{B_n} M^{\lambda}_q g \,\d\mm  = M^\lambda_q g(x). 
\end{equation}
We now divide both sides in \eqref{eqn:ali2} by $r_n=\sfd(x,y_n)$ and let $n\to\infty$. From \eqref{eqn:leb-points} and \eqref{eqn:slope} we get
$$ | \nabla u | (x) \leq 2 C M^\lambda_q g(x) + \lambda L .
$$
Since this inequality holds for $\mm$-a.e. $x$, we can choose an infinitesimal sequence $(\lambda_k)\subset (0,1)$ and use the
$\mm$-a.e. convergence of $M^{\lambda_k}_q g$ to $g$ to obtain \eqref{eqn:slope-weakgr-comp}.
\end{proof}

\begin{theorem}\label{thm:main1} Let $(X,\sfd,\mm)$ be a metric measure space with $\mm$ doubling,  
which supports a weak $(1,q)$-Poincar\'e inequality and satisfies $\supp\mm=X$. Then, 
for any open set $A\subset X$ it holds 
\begin{equation}\label{ts:lsc}
u_n,\,u\in\Lip_{\rm loc}(A),\,\,u_n\to u\,\,\,\text{in $L^1_{\rm loc}(A)$}\quad\Longrightarrow\quad
\liminf_{n\to\infty} \int_A |\nabla u_n|^q\,\d\mm\geq\int_A |\nabla u|^q \,\d\mm.
\end{equation}
In particular, understanding weak gradients according to \eqref{eq:locaopen},  it holds
 $|\nabla u|=|\nabla u |_{w,q}$ $\mm$-a.e. in $X$ for all $u\in\Lip_{\rm loc}(X)$.
\end{theorem}
\begin{proof} By a simple truncation argument we can assume that all functions $u_n$ are uniformly bounded, since
$|\nabla (M\land v\lor -M)|\leq |\nabla v|$ and $|\nabla (M\land v\lor -M)|\uparrow |\nabla v|$ as $M\to\infty$. 
Possibly extracting a subsequence we can also assume that the $\liminf$ in the right-hand side of \eqref{ts:lsc} is a limit and,
without loss of generality, we can also assume that it is finite. 
Fix a bounded open set $B$ with ${\rm dist}(B,X\setminus A)>0$ and let $\psi:X\to [0,1]$ be a cut-off Lipschitz function
identically equal to $1$ on a neighbourhood of $B$, with support bounded and contained in $A$. It is clear that the functions $v_n:=u_n\psi$ and
$v:=u\psi$ are globally Lipschitz, $v_n\to v$ in $L^q(X,\mm)$ and $(v_n)$ is bounded in $W^{1,q}(X,\sfd,\mm)$.

\null
From the reflexivity of this space proved in Corollary~\ref{cor:refl} we have that, possibly extracting a subsequence, $(v_n)$ 
weakly converges in the Sobolev space to a function $w$. Using Mazur's lemma, we construct another sequence $(\hat{v}_n)$ 
that is converging strongly to $w$ in $W^{1,q}(X,\sfd,\mm)$ and $\hat{v}_n$ is a finite convex combination of $v_n, v_{n+1} , \ldots $. In particular 
we get $\hat{v}_n\to w$ in $L^q(X,\mm)$ and this gives $w=v$. Moreover, 
$$ \int_B| \nabla \hat{v}_n|^q \,\d\mm \leq \sup_{k \geq n} \int_B | \nabla v_k|^q \,\d\mm.$$
Eventually, from Proposition~\ref{prop:slope-weak} applied to the functions $v - \hat{v}_n$ we get:
\begin{align*}
 \left( \int_B | \nabla v|^q \,\d\mm\right)^{1/q} & \leq \liminf_{n\to \infty} \left\{ \left( \int_B |\nabla \hat{v}_n |^q \,\d\mm\right)^{1/q} + 
 \left( \int_B |\nabla ( v - \hat{v}_n) |^q \, \d\mm \right)^{1/q}  \right\} \\
& \leq \limsup_{n \to \infty} \left\{ \left(\int_B | \nabla v_n|^q \,\d\mm \right)^{1/q}\right\} + C^{1/q}
 \limsup_{n \to \infty} \| v - \hat{v}_n \|_{W^{1,q}} \\
& =\limsup_{n\to \infty}  \left(\int_B | \nabla v_n|^q \,\d\mm \right)^{1/q}.
\end{align*}
Since $v_n\equiv u_n$ and $v\equiv u$ on $B$ we get
$$
\int_B | \nabla u|^q \,\d\mm\leq\limsup_{n\to\infty}\int_B|\nabla u_n|^q\,\d\mm\leq
\lim_{n\to\infty}\int_A|\nabla u_n|^q\,\d\mm
$$
and letting $B\uparrow A$ gives the result.
\end{proof}

\section{Appendix A: other notions of weak gradient}

In this section we consider different notions of weak gradients, all easily seen
to be intermediate between $|\nabla f|_{w,q}$ and $|\nabla f|_{*,q}$, and therefore
coincident, as soon as Theorem~\ref{thm:graduguali} is invoked.
These notions inspired those adopted in \cite{Ambrosio-Gigli-Savare11}.

\subsection{$q$-relaxed upper gradients and $|\nabla f|_{C,q}$}

In the relaxation procedure we can consider, instead of pairs $(f,\Lip_a f)$ 
(i.e. Lipschitz functions and their asymptotic Lipschitz constant),
pairs $(f,g)$ with $g$ upper gradient of $f$.

\begin{definition}[$q$-relaxed upper gradient]\label{def:cheeger00} We say that
$g\in L^q(X,\mm)$ is a $q$-relaxed upper gradient of $f\in
L^q(X,\mm)$ if there exist $\tilde{g}\in L^q(X,\mm)$, functions
$f_n\in L^q(X,\mm)$ and upper gradient $g_n$ of $f_n$ such that:
\begin{itemize}
\item[(a)] $f_n\to f$ in $L^q(X,\mm)$ and $g_n$ weakly converge to
$\tilde{g}$ in $L^q(X,\mm)$;
\item[(b)] $\tilde{g}\leq g$ $\mm$-a.e. in $X$.
\end{itemize}
We say that $g$ is a minimal $q$-relaxed upper gradient of $f$ if
its $L^q(X,\mm)$ norm is minimal among $q$-relaxed upper gradients.
We shall denote by $|\nabla f|_{C,q}$ the minimal $q$-relaxed upper
gradient.
\end{definition}

Again it can be proved (see \cite{Cheeger00}) that $|\nabla f|_{C,q}$ is local, and clearly
\begin{equation}\label{eq:weakgradients1}
|\nabla f|_{C,q}\leq |\nabla f|_{*,q}\qquad\text{$\mm$-a.e. in $X$}
\end{equation}
because any $q$-relaxed slope is a $q$-relaxed upper gradient. On the other hand,
the stability property of $q$-weak upper gradients stated in Theorem~\ref{thm:stabweak} gives
\begin{equation}\label{eq:weakgradients2}
|\nabla f|_{w,q}\leq |\nabla f|_{C,q}\qquad\text{$\mm$-a.e. in $X$.}
\end{equation}
In the end, thanks to Theorem~\ref{thm:graduguali}, all these notions coincide $\mm$-a.e. in $X$.

Notice that one more variant of the ``relaxed'' definitions is the one considered in \cite{Ambrosio-Gigli-Savare11},
with pairs $(f,|\nabla f|)$. It leads to a weak gradient intermediate between the ones on
\eqref{eq:weakgradients1}, but a posteriori equivalent, using once more Theorem~\ref{thm:graduguali}.

\subsection{$q$-upper gradients and $|\nabla f|_{S,q}$}

Here we recall a weak definition of upper gradient, taken from
\cite{Koskela-MacManus} and further studied in
\cite{Shanmugalingam00} in connection with the theory of Sobolev
spaces, where we allow for exceptions in \eqref{eq:uppergradient}.
This definition inspired the one given in \cite{Ambrosio-Gigli-Savare11}, based on test plans.

Recall that, for $\Gamma\subset AC([0,1],X)$, the $q$-modulus ${\rm
Mod}_q(\Gamma)$ is defined by 
\begin{equation}
\label{eq:defmod2} {\rm
Mod}_q(\Gamma):=\inf\Big\{\int_X\rho^q\,\d\mm: \ \int_\gamma\rho\geq
1\ \ \forall \gamma\in\Gamma\Big\}.
\end{equation}
We say that $\Gamma$ is ${\rm Mod}_q$-negligible if ${\rm
Mod}_q(\Gamma)=0$. Accordingly, we say that a Borel function
$g:X\to[0,\infty]$ is a $q$-upper gradient of $f$ if there exist a
function $\tilde f$ and a ${\rm Mod}_q$-negligible set $\Gamma$ such
that $\tilde{f}=f$ $\mm$-a.e. in $X$ and
\[
\big|\tilde f(\gamma_0)-\tilde f(\gamma_1)\big|\leq\int_\gamma
g\qquad\forall \gamma\in AC([0,1],X)\setminus\Gamma.
\]
It is not hard to prove that the collection of all  $q$-upper
gradients of $f$ is convex and closed, so that we can call minimal
$q$-upper gradient, and denote by $|\nabla f|_{S,q}$, the element
with minimal $L^q(X,\mm)$ norm. Furthermore, the inequality
\begin{equation}\label{allinequalities2}
|\nabla f|_{S,q}\leq|\nabla f|_{C,q}\qquad\text{$\mm$-a.e. in $X$}
\end{equation}
(namely, the fact that all $q$-relaxed upper gradients are $q$-upper
gradients) follows by a stability property of $q$-upper gradients
very similar to the one stated in Theorem~\ref{thm:stabweak} 
for $q$-weak upper gradients, see
\cite[Lemma~4.11]{Shanmugalingam00}. 

Observe that for a Borel set $\Gamma\subset C([0,1],X)$ and a test
plan $\ppi$, integrating on $\Gamma$ w.r.t. $\ppi$ the inequality
$\int_\gamma \rho\geq 1$ and then minimizing over $\rho$, we get
$$
\ppi(\Gamma)\leq (C(\ppi))^{1/q}\bigl({\rm
Mod}_q(\Gamma)\bigr)^{1/q}\biggl(\iint_0^1|\dot\gamma|^p\,\d
s\,\d\ppi(\gamma)\biggr)^{1/p},
$$
which shows that any ${\rm Mod}_q$-negligible set of curves is also
$q$-negligible according to Definition~\ref{def:testplans}. This
immediately gives that any $q$-upper gradient is a $q$-weak upper
gradient, so that
\begin{equation}\label{allinequalities3}
|\nabla f|_{w,q}\leq|\nabla f|_{S,q}\qquad\text{$\mm$-a.e. in $X$.}
\end{equation}
Combining \eqref{eq:weakgradients2},  \eqref{allinequalities2} and \eqref{allinequalities3}
we obtain that also $|\nabla f|_{S,q}$ coincides $\mm$-a.e. with all other gradients.

\section{Appendix B:  discrete gradients in general spaces}

Here we provide another type of approximation via discrete gradients which doesn't even require the space $(X,\sfd)$ to be doubling. 
We don't know whether this approximation can be used to obtain the reflexivity of $W^{1,q}(X,\sfd,\mm)$ even without doubling assumptions.

We slightly change the definition of discrete gradient: instead of taking the sum of the finite differences, that is forbidden due to the fact that the number of terms 
can not in general be uniformly bounded from above, we simply take the supremum among the finite differences. Let us fix a decomposition $A_i^\delta$ of
$\supp\mm$ as in 
Lemma~\ref{lemma:dec}. Let $u \in L^q(X, \mm)$ and denote by $u_{\delta,i}$ the mean of $u$ in $A_i^\delta$ as before. We consider the discrete gradient
$$ |\Dh{\delta} u |_{\infty} (x) =\frac 1{\delta} \sup_{A^\delta_j\sim A^\delta_i} \{ |u_{\delta,i} - u_{\delta,j}| \} \qquad \forall x \in A_i^\delta.$$

Then we consider the functional $\Fh{\delta}^{\infty}: L^q(X, \mm)\to [0,\infty]$ given by
 $$ \Fh{\delta}^{\infty} (u) := \int_X | \Dh{\delta} (u) |_{\infty}^q (x) \,\d\mm (x).$$
With these definitions, the following theorem holds.

\begin{theorem}\label{thm:mainnd}
Let $(X,\sfd,\mm)$ be a Polish metric measure space with $\mm$ finite on bounded sets.
Let $\F^{\infty}$ be a $\Gamma$-limit point of $\Fh{q,\delta}^{\infty}$ as $\delta\downarrow 0$, namely
$$ \F^{\infty}:= \Gamma \!- \! \lim_{k\to \infty} \mathcal{F}_{q,\delta_k}^{\infty}, $$
where $\delta_k\to 0$ and the $\Gamma$-limit is computed with respect to the $L^q(X,\mm)$-distance.
Then the functional $\F^{\infty}$ is equivalent to Cheeger's energy, namely there exists a constant $\eta_{\infty}=\eta_\infty(q)$ 
such that
\begin{equation}\label{ts:est-f-chnd}
\frac{1}{\eta_{\infty}} \Ch_q (u) \leq \F^{\infty} (u) \leq \eta_{\infty} \Ch_q (u)\qquad\forall u\in L^q(X,\mm).
\end{equation}
\end{theorem}

The proof follows closely the one of Theorem~\ref{thm:main}. An admissible choice for $\eta_{\infty}$ is $6^q$.
  
\section{Appendix C: some open problems}

In this section we list and discuss some open problems. 

\smallskip\noindent
{\bf 1. Optimality of the doubling assumption for reflexivity.} We don't know whether the doubling assumption on $(X,\sfd)$ can
be weakened. However, the finite difference scheme used in this paper seems really to rely on this assumption.

\smallskip\noindent
{\bf 2. Optimality of the Poincar\'e assumption for the lower semicontinuity of slope.} As shown to us by P.Koskela, the doubling
assumption, while sufficient to provide reflexivity of the Sobolev spaces $W^{1,q}(X,\sfd,\mm)$, is not sufficient to ensure the 
lower semicontinuity \eqref{lscChee} of slope. Indeed, one can consider for instance the Von Koch snowflake $X\subset\R^2$
endowed with the Euclidean distance. Since $X$ is a self-similar fractal satisfying Hutchinson's open set condition (see for
instance \cite{Falconer}), it follows
that $X$ is Ahlfors regular of dimension $\alpha=\ln 4/\ln 3\in (1,2)$, namely $0<\Haus{\alpha}(X)<\infty$, where $\Haus{\alpha}$ denotes $\alpha$-dimensional
Hausdorff measure in $\R^2$. Using self-similarity it is easy to check that $(X,\sfd,\Haus{\alpha})$ is doubling. However, since
absolutely continuous curves with values in $X$ are constant, the $q$-weak upper gradient of any Lipschitz function $f$ vanishes. Then, the equivalence
of weak and relaxed gradients gives $|\nabla f|_{*,q}=0$ $\Haus{\alpha}$-a.e. on $X$. By Proposition~\ref{prop:easy} we obtain Lipschitz functions
$f_n$ convergent to $f$ in $L^q(X,\Haus{\alpha})$ and satisfying
$$
\lim_{n\to\infty}\int_X{\rm Lip}_a^q(f_n,x)\,\d\Haus{\alpha}(x)=0.
$$
Since ${\rm Lip}_a(f_n,\cdot)\geq |\nabla f_n|$, if $|\nabla f|$ is not trivial we obtain a counterexample to \eqref{lscChee}.

One can easily show that any linear map, say $f(x_1,x_2)=x_1$, has a nontrivial slope on $X$ at least $\Haus{\alpha}$-a.e. in $X$. 
Indeed, $|\nabla f|(x)=0$ for some $x\in X$ implies that the geometric tangent space to $X$ at $x$, namely all limit points as $y\in X\to x$ of normalized
secant vectors $(y-x)/|y-x|$, is contained in the vertical line $\{x_1=0\}$. However, a geometric rectifiability criterion
(see for instance \cite[Theorem~2.61]{Ambrosio-Fusco-Pallara}) shows that this set of points $x$ is contained in a countable union
of Lipschitz curves, and it is therefore $\sigma$-finite with respect to $\Haus{1}$ and $\Haus{\alpha}$-negligible.

This proves that doubling is not enough.
On the other hand, quantitative assumptions weaker than the Poincar\'e inequality might still be sufficient to provide the
result.

\smallskip\noindent
{\bf 3. Dependence on $q$ of the weak gradient.} The dependence of $|\nabla f|_{w,q}$ on $q$ is still open: more
precisely, assuming for simplicity that $\mm(X)$ is finite, $f\in W^{1,q}(X,\sfd,\mm)$ easily implies via
Proposition~\ref{prop:easy} that $f\in W^{1,r}(X,\sfd,\mm)$ and that
$$
|\nabla f|_{r,*}\leq|\nabla f|_{q,*}\qquad\text{$\mm$-a.e. in $X$.}
$$
Whether equality $\mm$-a.e. holds or not is an open question. As pointed out to us by Gigli, this holds if $|\nabla g|_{w,q}$
is independent on $q$ for  a dense class $\mathcal D$ of functions (for instance Lipschitz functions $g$ with bounded support); indeed, if this the case, for 
any $g\in\mathcal D$ we have
$$
|\nabla f|_{q,*}\leq |\nabla g|_{q,*}+|\nabla (f-g)|_{q,*}=|\nabla g|_{r,*}+|\nabla (f-g)|_{q,*}
$$
and considering $g_n\in\mathcal D$ with $g_n\to f$ strongly in $W^{1,q}(X,\sfd,\mm)$ we obtain the result, since convergence occurs also
in $W^{1,r}(X,\sfd,\mm)$ and therefore $|\nabla g_n|_{r,*}\to|\nabla f|_{r,*}$ in $L^r(X,\mm)$. 

Under doubling and Poincar\'e assumptions, we know that these requirements are met with the class
$\mathcal D$ of Lipschitz functions with bounded support, therefore as pointed out
in \cite{Cheeger00} the weak gradient is independent of $q$. Assuming only the doubling condition, the question is
still open.

\def\cprime{$'$} \def\cprime{$'$}

\end{document}

%% file: macro2010.tex
\newcommand{\F}{\mathbb{F}}

\newcommand{\N}{\mathbb{N}}
\newcommand{\Q}{\mathbb{Q}}
\newcommand{\R}{\mathbb{R}}

\newcommand{\mm}{{\mbox{\boldmath$m$}}}

\newcommand{\ggamma}{{\mbox{\boldmath$\gamma$}}}

\newcommand{\ppi}{{\mbox{\boldmath$\pi$}}}

\newcommand{\sfd}{{\sf d}}

\newcommand{\rme}{{\mathrm e}}

\newcommand{\Kliminf}{K\kern-3pt-\kern-2pt\mathop{\rm lim\,inf}\limits}  
\newcommand{\supp}{\mathop{\rm supp}\nolimits}   
\newcommand{\Lip}{\mathop{\rm Lip}\nolimits}          
\renewcommand{\d}{{\mathrm d}}
\newcommand{\dt}{{\d t}}

\newcommand{\restr}[1]{\lower3pt\hbox{$|_{#1}$}}

\newcommand{\Haus}[1]{{\mathscr H}^{#1}}     
\newcommand{\Leb}[1]{{\mathscr L}^{#1}}      

\newcommand{\eps}{\varepsilon}  
\newcommand{\nchi}{{\raise.3ex\hbox{$\chi$}}}



%% file: macro_diffusion.tex
\newcommand{\Pc}[2]{\overline{#1}\kern-2pt^{\vphantom 0}_{#2}}

